\documentclass[12pt,reqno]{amsart}

\usepackage{amssymb}

\usepackage[all]{xy}

\numberwithin{equation}{section}

\newtheorem{theorem}{Theorem}[section]
\newtheorem{proposition}[theorem]{Proposition}
\newtheorem{lemma}[theorem]{Lemma}
\newtheorem{corollary}[theorem]{Corollary}

\theoremstyle{definition}
\newtheorem{definition}[theorem]{Definition}
\newtheorem{remark}[theorem]{Remark}

\usepackage{color}
\definecolor{darkred}{rgb}{1,0,0}
\definecolor{darkgreen}{rgb}{0,1,0}
\definecolor{darkblue}{rgb}{0,0,1}
\usepackage{hyperref}
\hypersetup{colorlinks,
linkcolor=darkblue,
filecolor=darkblue,
urlcolor=darkred,
citecolor=darkgreen}

\begin{document}

\baselineskip=15pt

\title[Pullback and direct image of parabolic connections and Higgs bundles]{Pullback and direct
image of parabolic connections and parabolic Higgs bundles}

\author[D. Alfaya]{David Alfaya}

\address{Department of Applied Mathematics and Institute for Research in Technology, ICAI School of Engineering,
Comillas Pontifical University, C/Alberto Aguilera 25, 28015 Madrid, Spain}

\email{dalfaya@comillas.edu}

\author[I. Biswas]{Indranil Biswas}

\address{School of Mathematics, Tata Institute of Fundamental
Research, Homi Bhabha Road, Mumbai 400005, India}

\email{indranil@math.tifr.res.in}



\date{}

\begin{abstract}
We provide an explicit algebraic construction for the pullback and direct image of parabolic bundles, parabolic Higgs bundles and parabolic 
connections through maps between Riemann surfaces. We show that these constructions preserve semistability and
polystability, and we prove that they are compatible with the nonabelian Hodge correspondence.
\end{abstract}

\maketitle

\tableofcontents

\section{Introduction}

Let $X$ be a compact connected Riemann surface, and let $D\,=\,\{x_1,\,\ldots,\,x_l\}$ be a finite subset of $X$. A parabolic bundle on $(X,
\,D)$ is a vector bundle $E$ on $X$ together with a weighted flag at the fiber $E_{x_i}$ for each $i\,=\,1,\,\ldots,\,l$, i.e., a
decreasing filtration by subspaces of $E_{x_i}$
$$E_{x_i}\,=\,E_i^1 \,\supsetneq\, E_i^2 \,\supsetneq \,\cdots \,\supsetneq \,E_i^{n_i} \,\supsetneq\, 0$$
together with an associated set of parabolic weights
$$0\,\le\, \alpha_i^1 \,< \,\alpha_i^2 \,<\,\cdots \,<\, \alpha_i^{n_i}\, .$$
Parabolic bundles were introduced by Mehta and Seshadri \cite{MS} for curves and a generalization for higher dimensional varieties
was later provided by Maruyama and Yokogawa \cite{MY}.

Let $K_X$ denote the canonical bundle of $X$. Given a parabolic vector bundle $E_*$ on $(X,\,D)$, a parabolic Higgs field on $E_*$ is an
$\mathcal{O}_X$-linear map
$$\theta\,:\,E\,\longrightarrow\, E\otimes K_X(D)$$
that preserves the parabolic filtration at the parabolic points. Such a pair $(E_*,\,\theta)$ is called a parabolic Higgs bundle on $(X,\,
D)$. Analogously, a parabolic connection on $E_*$ is a logarithmic connection on $E$ with poles over $D$ which preserves the
filtration, i.e., it is a $\mathbb{C}$-linear map
$$\nabla\,:\,E\,\longrightarrow \,E\otimes K_X(D)$$
satisfying the Leibniz rule, which says that
$$\nabla(fs)\,=\,f\nabla(s)+s\otimes \partial f$$
for all local holomorphic functions $f$ and all local holomorphic sections $s$ of $E$, such that the residue at each parabolic point
$$\text{Res}(\nabla,\,x_i)\,:\, E_{x_i}\,\longrightarrow\, E_{x_i}$$
preserves the parabolic filtration (more details are in Section \ref{se2}). These filtered logarithmic analogues of Higgs bundles
and connections where introduced by Simpson \cite{Si} as the natural objects which arise in the process of
extending the classical nonabelian Hodge correspondence between Higgs bundles, connections and $\operatorname{GL}(r,
\mathbb{C})$-representations of the fundamental group $\pi_1(X)$ of a Riemann surface $X$
to the noncompact case. In particular, strongly parabolic Higgs bundles, which are parabolic bundles with nilpotent residue at the parabolic points, are the natural objects which correspond to representations of $\pi_1(X\backslash D)$ in $\operatorname{GL}(r,\mathbb{C})$.

Let us suppose that we have a marked smooth projective curve $(X,\,D)$ as before, together with either a parabolic bundle $E_*$, or a parabolic Higgs bundle
$(E_*,\,\theta)$ or a parabolic connection $(E_*,\,\nabla)$ on $(X,D)$. Let $f\,:\,Y\,\longrightarrow\, X$ and
$\phi\,:\,X\, \longrightarrow\, Z$ be nonconstant holomorphic maps between compact Riemann surfaces.
In this work we study the pullbacks to $Y$ and direct images to $Z$ of these types of 
parabolic objects on $(X,\, D)$ through these maps.

Pullbacks and direct images of parabolic objects have been explored in the literature before in different contexts. Using the equivalence 
between parabolic bundles on a curve and orbifold bundles on an associated root stack \cite{Bi2}, Dhillon and Joyner \cite{DJ} have explored 
pullbacks of parabolic bundles between curves and studied in depth the pullbacks to covers of $\mathbb{P}^1\backslash \{0,1,\infty\}$. Kumar 
and Majumder \cite{KM} defined pullbacks and direct images in positive characteristic based on the same correspondence. In \cite{BM}, 
explicit descriptions of the direct images of a parabolic bundle or a parabolic connection for a map between curves were provided and in 
\cite{FL} the pullback of parabolic Higgs bundles through degree two maps between curves was studied.

In this paper we provide explicit algebraic descriptions for the pullbacks and direct images of parabolic vector bundles, parabolic 
Higgs bundles and parabolic connections through any nonconstant (possibly ramified) map of smooth projective curves and we study their stability. The main
results of this work can then be summarized in the following theorem:

\begin{theorem}[Lemma \ref{lemn1}, Proposition \ref{prop4}, Lemma \ref{lem2}, Theorem \ref{thm1}, Proposition \ref{prop3}, Proposition \ref{pr1} and Proposition \ref{prop:semistableConn}]
Let $E_*$ be a parabolic bundle on $(X,\, D)$, and let $(E_*,\,\theta)$ be a parabolic Higgs bundle. Let $(E_*,\,\nabla)$ be a parabolic connection
on $(X,\,D)$. Let $f:Y\longrightarrow X$ and $\phi:X\longrightarrow Z$ be two nonconstant maps between connected smooth
complex projective curves. Then the following statements hold:
\begin{enumerate}
\item $f^*E_*$, $f^*(E_*,\,\theta)$ and $f^*(E_*,\,\nabla)$ are semistable if and only if $E_*$, $(E_*,\,\theta)$ and $(E_*,\, \nabla)$
are semistable respectively.

\item $\phi_*E_*$, $\phi_*(E_*,\,\theta)$ and $\phi_*(E_*,\,\nabla)$ are semistable if and only if $E_*$, $(E_*,\,\theta)$ and
$(E_*,\,\nabla)$ are semistable respectively.

\item $f^*E_*$ is polystable if and only if $E_*$ is polystable.

\item If $E_*$ is polystable then $\phi_*E_*$ is polystable.
\end{enumerate}
\end{theorem}

Regarding the compatibility between the pullback and direct image and the nonabelian Hodge correspondence, Donagi, Pantev and Simpson 
\cite{DPS} have provided descriptions of higher direct images and proven their compatibility with nonabelian Hodge correspondence when a map 
from a surface to a curve is considered. The techniques in \cite{DPS} can also be extended to a broader framework, and it is also mentioned 
in the paper that, in general, the pullback and pushforward operations are well-defined for harmonic bundles and local systems and that they 
commute with nonabelian Hodge correspondence. Nevertheless, the case of a ramified map between algebraic curves has not been treated 
explicitly in the literature. We show explicitly that the algebraic constructions for the pullback and direct image provided in this work 
agree with the ones arising metrically from the acceptable metrics which are part of the tame harmonic bundles involved in the nonabelian 
Hodge theory for noncompact curves \cite{Si}. We prove the following result showing that taking the pullback or direct image of parabolic 
Higgs bundles and parabolic connections commute with the nonabelian Hodge correspondence.

\begin{theorem}[Theorem \ref{thm:pullbackNAHT} and Theorem \ref{thm:directNAHT}]
Let $(E_*,\,\theta)$ be a parabolic Higgs bundle, and let $(E_*,\,\nabla)$ be a parabolic connection on $(X,\,D)$ which are related by
the nonabelian Hodge correspondence. Let $f\,:\,Y\,\longrightarrow\, X$ and $\phi\,:\,X\,\longrightarrow\, Z$ be two nonconstant map
between connected smooth complex projective curves. Then the following statements hold:
\begin{enumerate}
\item $f^*(E_*,\,\theta)$ is a parabolic Higgs bundle, and $f^*(E_*,\,\nabla)$ is a parabolic connection, on $(Y,\,f^{-1}(D)_{\text{red}})$ which 
are related through nonabelian Hodge correspondence.

\item $\phi_*(E_*,\,\theta)$ is a parabolic Higgs bundle, and $\phi_*(E_*,\,\nabla)$ is a parabolic connection on $(Z,\,\Delta)$, which
are related through nonabelian Hodge correspondence, where $\Delta$ is the image of the union of $D$ and the ramification locus of $\phi$.
\end{enumerate}
\end{theorem}

Tables of the expected parabolic weights and eigenvalues of the resulting parabolic objects have also been computed (see Table 
\ref{table:pullback} and Table \ref{table:direct}).

The paper is organized as follows. The main definitions and notations about parabolic bundles, parabolic Higgs bundles and parabolic 
connections are introduced in Section \ref{se2}. The pullback and direct image of parabolic bundles and some of their main properties are 
described through Sections \ref{se3} and \ref{se4}. Section \ref{section:pullback} is devoted to the construction of the pullback of 
parabolic Higgs bundles and parabolic connections and the study of its main properties, including the main polystability preservation 
results of this work (Theorem \ref{thm1} and Proposition \ref{prop3}). Analogously, Section \ref{section:directimage} is devoted to the 
direct image of parabolic Higgs bundles and parabolic connections. Finally, the compatibility between the previous constructions and 
nonabelian Hodge Theory is addressed in Section \ref{section:naht}.

\section{Parabolic Higgs bundles and parabolic connections}
\label{se2}

\subsection{Parabolic Higgs bundles}\label{se2.1}

Let $X$ be a compact connected Riemann surface. The holomorphic cotangent bundle of $X$ will be denoted by $K_X$.
Let
$$
D\,\,:=\,\, \{x_1,\, \cdots,\, x_\ell\}\,\, \subset\,\, X
$$
be a finite subset. The divisor $\sum_{i=1}^\ell x_i$ will also be denoted by $D$. For a holomorphic vector bundle
$V$ on $X$, the vector bundle $V\otimes {\mathcal O}_X(D)$ will be denoted by $V(D)$.

Take a holomorphic vector bundle $E$ on $X$. A quasiparabolic structure on $E$ is a
strictly decreasing filtration of subspaces
\begin{equation}\label{e1}
E_{x_i}\,=\, E^1_i\, \supsetneq\, E^2_i \,\supsetneq\, \cdots\, \supsetneq\,
E^{n_i}_i \, \supsetneq\, E^{n_i+1}_i \,=\, 0
\end{equation}
for every $1\, \leq\, i\, \leq\, \ell$; here $E_{x_i}$ denotes the fiber
of $E$ over the point $x_i\,\in\, D$. A \textit{parabolic structure} on $E$ is a
quasiparabolic structure as above together with $\ell$ increasing sequences of real numbers
\begin{equation}\label{e2}
0\, \leq\, \alpha^1_i\, <\, \alpha^2_i\, <\,
\cdots\, < \, \alpha^{n_i}_i \, < 1\, , \ \ 1\, \leq\, i\, \leq\, \ell \, ;
\end{equation}
the real number $\alpha^j_i$ is called the parabolic weight of the subspace $E^j_i$ in
the quasiparabolic filtration in \eqref{e1}. To clarify, $j$ in $\alpha^j_i$ is an index
and not an exponent. The multiplicity of a parabolic weight $\alpha^j_i$ at
$x_i$ is defined to be the dimension of the complex vector space $E^j_i/E^{j+1}_i$.
A parabolic vector bundle is a holomorphic vector bundle with a parabolic structure.

The \textit{parabolic degree} of a parabolic vector bundle $E_*\,:=\, \left(E,\, \{E^j_i\}, \,\{\alpha^j_i\}\right)$
is defined to be
$$
\text{par-deg}(E_*)\,=\, \text{degree}(E)+\sum_{i=1}^\ell \sum_{j=1}^{n_i}
\alpha^j_i\cdot\dim \left(E^j_i/E^{j+1}_i\right)
$$
\cite[p.~214, Definition~1.11]{MS}, \cite[p.~78]{MY}. The real number
$$
\mu(E_*)\,\, :=\, \frac{\text{par-deg}(E_*)}{\text{rank}(E_*)}
$$
is called the slope of $E_*$.

Take any holomorphic subbundle $F\, \subset\, E$. For each $x_i\, \in\, D$, the
fiber $F_{x_i}$ has a filtration of subspaces obtained by intersecting the quasiparabolic filtration
of $E_{x_i}$ with the subspace $F_{x_i}\, \subset\, E_{x_i}$. The parabolic weight of a subspace $B\,
\subset\, F_{x_i}$ occurring in this filtration of subspaces of $E_{x_i}$ is the maximum of the numbers
$$\{\alpha^j_i\, \,\mid\,\, B\, \subset \, E^j_i\cap F_{x_i}\}\, .$$
This way, the parabolic structure on $E$ produces a parabolic structure on the subbundle $F$.
The resulting parabolic bundle will be denoted by $F_*$.

A parabolic vector bundle $$E_*\,=\, \left(E,\, \{E^j_i\}, \,\{\alpha^j_i\}\right)$$ is called
\textit{stable} (respectively, \textit{semistable}) if for all holomorphic subbundles
$F\, \subsetneq\, E$ of positive rank the following inequality holds:
$$
\mu(F_*)\, <\, \mu(E_*)\ \ \left(\text{respectively, }\, \ \mu(F_*) \, \leq\, \mu(E_*)\right).
$$

A parabolic vector bundle is called \textit{polystable} if it is a direct sum of stable
parabolic vector bundles of same parabolic slope.

Take a parabolic bundle $E_*\,:=\, \left(E,\, \{E^j_i\}, \,\{\alpha^j_i\}\right)$. Let
\begin{equation}\label{e3}
\text{End}_P(E_*)\, \subset\, \text{End}(E)\,=\, E\otimes E^*
\end{equation}
be the coherent analytic subsheaf defined by the following condition: A holomorphic
section $s\, \in\, H^0(U,\, \text{End}(E)\big\vert_U)$, where $U\, \subset\, X$ is any
open subset, is a section of $\text{End}_P(E_*)\big\vert_U$ if $s(E^j_i)\, \subset\, E^j_i$
for all $x_i\, \in\, U$ and all $1\, \leq\, j\, \leq\, n_i$. Let
\begin{equation}\label{e3a}
\text{End}_n(E_*)\, \subset\, \text{End}_P(E_*)
\end{equation}
be the coherent analytic subsheaf defined by the following condition: A holomorphic
section $s\, \in\, H^0(U,\, \text{End}_P(E_*)\big\vert_U)$, where $U\, \subset\, X$ is any
open subset, is a section of $\text{End}_n(E_*)\big\vert_U$ if $s(E^j_i)\, \subset\, E^{j+1}_i$
for all $x_i\, \in\, U$ and all $1\, \leq\, j\, \leq\, n_i$.

A Higgs field on $E_*$ is a holomorphic section
$$
\theta\, \in\, H^0(X,\, \text{End}_P(E_*)\otimes K_X(D))
$$
(see \eqref{e3}).
If a Higgs field $\theta$ is section of $\text{End}_n(E_*)\otimes K_X(D)$ (see \eqref{e3a}), then it is
called a strongly parabolic Higgs field. A \textit{parabolic Higgs bundle} is a parabolic bundle equipped with a Higgs field.
A \textit{strongly parabolic Higgs bundle} is a parabolic Higgs bundle such that the Higgs field is strongly parabolic.

A parabolic Higgs bundle $(E_*,\, \theta) \,=\, \left(\left(E,\, \{E^j_i\}, \,\{\alpha^j_i\}\right),\, \theta\right)$
is called \textit{stable} (respectively, \textit{semistable}) if for all holomorphic subbundles
$F\, \subsetneq\, E$ of positive rank satisfying the condition that $\theta(F)\, \subset\, F\otimes K_X(D)$
the following inequality holds:
$$
\mu(F_*)\, <\, \mu(E_*)\ \ \left(\text{respectively, }\, \ \mu(F_*) \, \leq\, \mu(E_*)\right).
$$

A parabolic Higgs bundle is called \textit{polystable} if it is a direct sum of stable
parabolic Higgs bundles of same parabolic slope.

A strongly parabolic Higgs bundle is semistable, stable or polystable if it is
semistable, stable or polystable as a parabolic Higgs bundle.

\subsection{Parabolic connections}\label{se2.2}

Take $X$ and $D$ as in Section \ref{se2.1}. Let $V$ be a holomorphic
vector bundle on $X$. A \textit{logarithmic connection} on $V$ singular over
$D$ is a holomorphic differential operator
$$
{\nabla}\, :\, V\, \longrightarrow\, V\otimes K_X(D)
$$
satisfying the Leibniz identity which states that
\begin{equation}\label{e4}
{\nabla}(fs)\,=\, f{\nabla}(s)+ s\otimes df
\end{equation}
for any locally
defined holomorphic function $f$ on $X$ and any locally defined holomorphic section $s$ of $V$ (see
\cite{De}, \cite{At}).

The fiber of $K_X(D)$ over any $y\, \in\, D$ is identified with $\mathbb C$ by the Poincar\'e adjunction
formula \cite[p.~146]{GH}. To explain this isomorphism
\begin{equation}\label{pa}
K_X(D)_y \, \stackrel{\sim}{\longrightarrow}\, {\mathbb C}\, ,
\end{equation}
let $z$ be a holomorphic coordinate function on $X$ defined on an analytic
open neighborhood of $y$ such that $z(y)\,=\, 0$. Then we have an isomorphism
${\mathbb C}\, \longrightarrow\, K_X(D)_y$ that sends any $c\, \in\, \mathbb C$
to $c\cdot \frac{dz}{z}(y)\,\in\, K_X(D)_y$. It is straightforward to check that this map
${\mathbb C}\, \longrightarrow\, K_X(D)_y$ is actually independent of the choice
of the holomorphic coordinate function $z$.

Let ${\nabla}_V\, :\, V\, \longrightarrow\, V\otimes K_X(D)$ be a logarithmic connection
on $V$. From \eqref{e4} it follows that the composition of homomorphisms
$$
V\, \xrightarrow{\,\ {\nabla}_V\,\ }\, V\otimes K_X(D) \, \longrightarrow\,
(V\otimes K_X(D))_y\,\stackrel{\sim}{\longrightarrow}\, V_y
$$
is ${\mathcal O}_X$--linear; the above isomorphism $(V\otimes K_X(D))_y\,
\stackrel{\sim}{\longrightarrow}\, V_y$ is given by the isomorphism in
\eqref{pa}. Therefore, this composition of homomorphisms produces a
$\mathbb C$--linear homomorphism
$$
{\rm Res}({\nabla}_V,\,y)\, :\, V_y\, \longrightarrow\, V_y\, ,
$$
which is called the \textit{residue} of ${\nabla}_V$ at $y$; see \cite{De}.

Take a parabolic vector bundle $E_*\,=\, \left(E,\, \{E^j_i\}, \,\{\alpha^j_i\}\right)$.

\begin{definition}\label{dlc}
A \textit{connection} on $E_*$ is a logarithmic connection ${\nabla}$ on $E$, singular over
$D$, such that $\text{Res}(D,\,x_i)(E^j_i)\, \subset\, E^j_i$ for all $1\,\leq\, j\,\leq\, n_i$,
$1\,\leq\, i\, \leq\, \ell$ (see \eqref{e1}).
\end{definition}

In the literature, it is sometimes also required that the endomorphism of $E^j_i/E^{j+1}_i$ induced by $\text{Res}(D,\,x_i)$ coincides with 
multiplication by the parabolic weight $\alpha^j_i$ for all $1\,\leq\, j\,\leq\, n_i$, $1\,\leq\, i\, \leq\, \ell$ (see 
\cite[Section~2.2]{BL}). Parabolic connections satisfying this additional restriction correspond through the nonabelian Hodge correspondence 
to strongly parabolic Higgs bundles. In this work, we will not assume that this condition is satisfied, but we will later prove that 
pullbacks and direct images of connections satisfying this additional ``residual'' condition also satisfy a corresponding ``residual'' 
condition. The details are in Remark \ref{rmk:pullback} and Remark \ref{rmk:direct}.

\section{Pullback of parabolic bundles}\label{se3}

Take $(X,\, D)$ as before. Let
\begin{equation}\label{e5}
f\, :\, Y\, \longrightarrow\, X
\end{equation}
be a nonconstant holomorphic map from a compact connected Riemann surface. For each
$x_i\, \in\, D$, let
\begin{equation}\label{e6}
f^{-1}(x_i)_{\rm red}\,=\, \{y_{i,1},\, \cdots,\, y_{i,b_i}\}\, \subset\, Y
\end{equation}
be the set-theoretic inverse image. The divisor $\sum_{j=1}^{b_i} y_{i,j}$ on $Y$ will also be denoted 
by $f^{-1}(x_i)_{\rm red}$. Define the subset
\begin{equation}\label{e7}
B\,\,:=\,\, \bigcup_{i=1}^\ell f^{-1}(x_i)_{\rm red}\,\,=\,\, f^{-1}(D)_{\rm red}\, \subset\, Y .
\end{equation}
The divisor $\sum_{i=1}^\ell f^{-1}(x_i)_{\rm red}$ will also be denoted by $B$.

Take a parabolic vector bundle $E_*\,:=\, \left(E,\, \{E^j_i\}, \,\{\alpha^j_i\}\right)$ on
$X$ with parabolic structure over $D$. We will construct a pulled back parabolic vector bundle
$f^*E_*$ on $Y$ with parabolic structure over the divisor $B$ in \eqref{e7}.

We first consider the special case where $\text{rank}(E)\,=\, 1$. So for each $x_i\, \in\, D$ the
parabolic weight of $E_*$ is $\alpha_{i,1}\,=:\, \alpha_i\, \geq\, 0$. For any $1\, \leq\, i\, \leq\, \ell$,
and $1\,\leq\, j\, \leq\, b_i$, let $m_{i,j}\, \geq\, 1$ be the multiplicity of $f$ at the point $y_{i,j}$ 
(see \eqref{e6}).
For any $\lambda\, \in\, \mathbb R$, let $\lfloor{\lambda}\rfloor$ be
the integral part of $\lambda$, so we have $0\, \leq\, \lambda-\lfloor{\lambda}\rfloor \, <\, 1$.

The holomorphic line bundle on $Y$ underlying the parabolic line bundle $f^*E_*$ is
\begin{equation}
\label{eq:parpullback}
F\, :=\, (f^*E)\otimes{\mathcal O}_Y(\sum_{i=1}^\ell\sum_{j=1}^{b_i} \lfloor{m_{i,j}\alpha_i}\rfloor
\cdot y_{i,j})\, ,
\end{equation}
and the parabolic weight of $F_{y_{_{i,j}}}$ is $m_{i,j}\alpha_i -
\lfloor{m_{i,j}\alpha_i}\rfloor$. From this definition it follows immediately that
\begin{equation}\label{d1}
\text{par-deg}(f^*E_*)\,=\, \text{degree}(f)\cdot \text{par-deg}(E_*)\, .
\end{equation}

Note that in the above construction of $f^*E_*$ we did not use that $X$ is compact.
More precisely, the above construction remains valid if $X \,=\, \overline{X}\setminus S$,
where $\overline{X}$ is a compact connected Riemann surface and $S\, \subset\, \overline{X}$
is a finite subset.

Now let $E_*$ be a parabolic vector bundle on $X$ of rank $r\,\geq\, 2$.
Any parabolic vector bundle $E_*$ can locally be expressed as a direct sum of parabolic
line bundles. In other words, $X$ can be covered by Zariski open subsets $U_1,\,
\cdots,\, U_m$ such that $E_*\big\vert_{U_j}$ is a direct sum of parabolic line bundles
on $U_j$ for all $1\, \leq\, j\, \leq\, m$. Let
\begin{equation}\label{d2}
E_*\big\vert_{U_j}\,=\, \bigoplus_{k=1}^r L(j,k)_*
\end{equation}
be the decomposition of
$E_*\big\vert_{U_j}$ into a direct sum of parabolic line bundles. Let
$$
f_j\,:=\, f\big\vert_{f^{-1}(U_j)} \, :\, f^{-1}(U_j)\, \longrightarrow\, U_j
$$
be the restriction of $f$ to $f^{-1}(U_j)$. Now define
\begin{equation}
\label{eq:parpullback2}
f^*_j(E_*\big\vert_{U_j})\,=\, f^*_j\left(\bigoplus_{k=1}^r L(j,k)_*\right)
\,:=\, \bigoplus_{k=1}^r f^*_j L(j,k)_*
\end{equation}
to be the direct sum of pull backs of parabolic line bundles (see \eqref{d2}). Consequently,
we get a parabolic vector bundle $f^*_j(E_*\big\vert_{U_j})$ over $f^{-1}(U_j)$
with parabolic structure over $B\bigcap f^{-1}(U_j)$ (see \eqref{e7}).

$1\, \leq\, j,\, j'\, \leq\, m$, the two parabolic vector bundles
$f^*_j(E_*\big\vert_{U_j})$ and $f^*_{j'}(E_*\big\vert_{U_{j'}})$ are canonically
identified over $f^{-1}(U_j\bigcap U_{j'})$. This identification is obtained
from the identity map of $(f^*E)\big\vert_{U_j\cap U_{j'}}$. Therefore, these
parabolic vector bundles $f^*_j(E_*\big\vert_{U_j})$, $1\, \leq\, j\, \leq\, m$, patch together
compatibly to define a parabolic vector
bundle $f^*E_*$ on $Y$ with parabolic structure over $B$.

\begin{lemma}\label{lemn1}
Let $Z$ be a compact connected Riemann surface and $\psi\, :\, Z\, \longrightarrow\, Y$
a surjective holomorphic map. Then for any parabolic vector bundle $E_*$ on $X$,
$$
(f\circ\psi)^*E_*\,=\, \psi^*(f^*E_*)_*\, .
$$
\end{lemma}

\begin{proof}
This is a straightforward consequence of the construction of the pullback of a parabolic bundle.
\end{proof}

\begin{lemma}\label{lemn2}\mbox{}
\begin{enumerate}
\item The equality ${\rm{par\text{-}deg}}(f^*E_*)\,=\, {\rm degree}(f)\cdot {\rm{par\text{-}deg}}(E_*)$ holds.

\item The parabolic bundle $E_*$ is parabolic semistable if and only if $f^*E_*$ is parabolic semistable.
\end{enumerate}
\end{lemma}

\begin{proof}
Statement (1) follows immediately from the construction of $f^*E_*$ and \eqref{d1}.

To prove (2), first note that if $E_*$ is not parabolic semistable, then $f^*E_*$ is not parabolic
semistable. Indeed, if $F_*\, \subset\, E_*$ contradicts the semistability condition for $E_*$, then
using Statement (1) it follows that $f^*F_*$ contradicts the semistability condition for $f^*E_*$.

To prove the converse, in view of Lemma \ref{lemn1} and Statement (1) it suffices to consider the case where $f$ is
a (ramified) Galois covering. In that case, if $f^*E_*$ is not semistable, then the first nonzero term of the
Harder-Narasimhan filtration of $f^* E_*$ (also called that maximal semistable subbundle of $f^*E_*$) is preserved
under the action, on $f^*E_*$, of the Galois group $\text{Gal}(f)$ for $f$, and hence it
is of the form $f^*F_*$, where $F_*$ is a subbundle $F\, \subset\, E$
equipped with the parabolic structure induced by $E_*$. Now from Statement (1) it follows that
$F_*$ contradicts the semistability condition for $E_*$.
\end{proof}

\section{Direct image of parabolic bundles}\label{se4}

Let $Z$ be a compact connected Riemann surface and
\begin{equation}\label{ep}
\phi\,:\, X\, \longrightarrow\, Z
\end{equation}
a nonconstant holomorphic map.
Let
\begin{equation}\label{e8}
R\, \subset\, X
\end{equation}
be the ramification locus of $\phi$. To clarify, we do not assume that $R$ and $D$ are disjoint.
For any point $x\,\in\, X$, let $m_x\, \geq\, 1$ be the multiplicity of $\phi$ at $x$,
so $m_x\, \geq\, 2$ if and only if $x\, \in\, R$. Define the finite subset
\begin{equation}\label{e9}
\Delta\,=\, \phi (R\cup D) \, \subset\, Z.
\end{equation}

Take a parabolic vector bundle $E_*\,:=\, \left(E,\, \{E^j_i\}, \,\{\alpha^j_i\}\right)$ on
$X$ with parabolic structure on $D$. We will construct a parabolic structure on the direct image
$\phi_*E$ whose parabolic divisor is $\Delta$ in \eqref{e9}.

As in Section \ref{se3}, we first assume that ${\rm rank}(E)\,=\, 1$.
For each $x_i\, \in\, D$ the parabolic weight of $E_*$ is $\alpha_{i,1}\,=:\, \alpha_{x_i}$. We will
extend the parabolic points of $E_*$ in such that it does not change $E_*$.
For any $x\, \in\, R\setminus (R\cap D)$, we equip $E_x$ with the trivial parabolic structure,
meaning $E_x$ is given the parabolic weight $\alpha_x\,=\, 0$. This does not alter the parabolic structure of $E_*$,
but it helps in describing the parabolic structure of $\phi_*E$.

We now recall a general property of a direct image. Take a holomorphic vector bundle $F$ on $X$, and
consider the direct image $\phi_*F$. For any point $y\, \in\, Z$, the fiber $(\phi_* F)_y$
of $\phi_* F$ over $y$ has a canonical decomposition
\begin{equation}\label{e10}
(\phi_* F)_y\,=\, \bigoplus_{x\,\in \,\phi^{-1}(y)} V_x\, .
\end{equation}
We will describe the subspaces $V_x$ of $(\phi_* F)_y$. Take an analytic simply connected open neighborhood
$x\, \in\, U\, \subset\, X$ of $x$ such that
\begin{itemize}
\item $U\bigcap \phi^{-1}(y)\,=\, \{x\}$,

\item $U\bigcap R\, \subset\, \{x\}$, and

\item $\# \phi^{-1}(y')\bigcap U\,=\, m_x$ (recall that 
$m_x\, \geq\, 1$ is the multiplicity of $\phi$ at $x$) for all $y'\, \in\, \phi(U)\setminus \{y\}$.
\end{itemize}
Let
\begin{equation}\label{e12}
\psi\, :=\, \phi\big\vert_U \, :\, U \, \longrightarrow\, \phi(U)
\end{equation}
be the restriction of $\phi$ to $U$. The fiber $(\psi_* (F\big\vert_U))_y$ is evidently a subspace of $(\phi_* F)_y$. This
subspace will be denoted by $V_x$. Now we have the decomposition in \eqref{e10}.

Consider the decomposition
\begin{equation}\label{e11}
(\phi_* E)_y\,=\, \bigoplus_{x\,\in\, \phi^{-1}(y)} V_x
\end{equation}
in \eqref{e10} for the vector bundle $E$. The parabolic structure on $(\phi_* E)_y$ will be described by
giving a parabolic structure on each $V_x$ and then taking their direct sum.

To give a parabolic structure on $V_x$, first note that for any $j\, \geq\, 0$, there is a natural injective homomorphism
of coherent analytic sheaves
$$
\psi_* (E\otimes {\mathcal O}_U(-j x)) \, \hookrightarrow\, \psi_* E\, ,
$$
where $\psi$ is the map in \eqref{e12}. The image of $(\psi_* (E\otimes {\mathcal O}_U(-j x)))_y$ in
$(\phi_* E)_y$ by this homomorphism will be denoted by $(\psi_* (E\otimes {\mathcal O}_U(-j x)))'_y$. 
We have a filtration of subspaces of $V_x$:
\begin{equation}\label{e13}
V_x\,:=\, (\psi_* E)_y \, \supset \, (\psi_* (E\otimes {\mathcal O}_U(-x)))'_y
\, \supset \, (\psi_* (E\otimes {\mathcal O}_U(-2x)))'_y
\end{equation}
$$
\supset\, \cdots\, \supset\,
(\psi_* (E\otimes {\mathcal O}_U(-(m_x-1)x)))'_y\, \supset\, (\psi_* (E\otimes {\mathcal O}_U(-m_x x)))'_y\,=\,0,
$$
where $m_x\, \geq\, 1$ is the multiplicity of $\phi$ at $x$; note that $\psi_* (E\otimes {\mathcal O}_U(-m_x x))\,=\,
(\psi_* E)\otimes {\mathcal O}_{\psi (U)}(-y)$ (projection formula), and hence
$(\psi_* (E\otimes {\mathcal O}_U(-m_x x)))'_y\,=\,0$. The parabolic weight of the above subspace
$$
(\psi_* (E\otimes {\mathcal O}_U(-kx)))'_y\, \subset\, V_x
$$
in \eqref{e13} is $\frac{k+\alpha_x}{m_x}$.

Therefore, we have a parabolic structure on the vector space $V_x$. Now taking the direct sum of these parabolic structures
we get a parabolic structure on $E_y$ using \eqref{e11}. Note that the parabolic weight of $(\psi_* (E\otimes {\mathcal O}_U(-kx)))'_y$
is nonzero even if $\alpha_x\,=\, 0$.

Next let $E_*$ be a parabolic vector bundle on $X$ of rank $r\,\geq\, 2$. As in Section \ref{se3} we can reduce this to
the case of line bundles by decomposing $E_*$ into a direct sum of parabolic line bundles over suitable open subsets of $X$.

Choose connected open subsets ${\mathbb U}_j\, \subset\, X$, $1\,\leq\, j\, \leq\, m$, such that
\begin{enumerate}
\item $\bigcup_{j=1}^m {\mathbb U}_j\, =\, X$,

\item $\phi^{-1}(\phi({\mathbb U}_j))\,=\, {\mathbb U}_j$ for every $1\, \leq\, j\, \leq\, m$, and

\item we are given a decomposition of $E_*\big\vert_{{\mathbb U}_j}$ into a direct sum of parabolic line bundles
\begin{equation}\label{e14}
E_*\big\vert_{{\mathbb U}_j}\,=\, \bigoplus_{k=1}^r L(j,k)_*\, .
\end{equation}
\end{enumerate}
The holomorphic line bundle on ${\mathbb U}_j$ underlying $L(j,k)_*$ will be denoted $L(j,k)$.

Let
$$
\phi^j\,:=\, \phi\big\vert_{{\mathbb U}_j} \, :\, {\mathbb U}_j\, \longrightarrow\, \phi({\mathbb U}_j)
$$
be the restriction of $\phi$ to ${\mathbb U}_j$. Using \eqref{e14} we have
\begin{equation}\label{e15}
(\phi_*E)\big\vert_{\phi({\mathbb U}_j)}\,\,=\,\, \bigoplus_{k=1}^r \phi^j_* L(j,k)
\end{equation}
for every $1\, \leq\, j\, \leq\, m$.

We already constructed a parabolic structure on $\phi^j_* L(j,k)$ using the
parabolic structure of $L(j,k)_*$. Therefore, using \eqref{e15} we get
a parabolic structure on $(\phi_*E)\big\vert_{\phi({\mathbb U}_j)}$. 
For every $1\, \leq\, j,\, j'\, \leq\, m$, the parabolic structures on
$(\phi_*E)\big\vert_{\phi({\mathbb U}_j)}$ and $(\phi_*E)\big\vert_{\phi({\mathbb U}_{j'})}$
coincide over $\phi({\mathbb U}_j\bigcap {\mathbb U}_{j'})$. Therefore, we obtain a parabolic
structure on $\phi_*E$.

\begin{lemma}\label{lem-n}
The parabolic degree of $\phi_*E$ equipped with the above parabolic structure coincides with the
parabolic degree of $E_*$.
\end{lemma}

\begin{proof}
When $E_*$ is of rank one, this follows directly from the construction of the parabolic structure on $\phi_*E$.

The general case of higher ranks will be reduced to the case of rank one. For a parabolic vector bundle
$V_*$ of rank $r$, let $\det V_*\,=\, \bigwedge^r V_*$ be the parabolic wedge product; see \cite{Yo}, \cite{Bi1}
for parabolic tensor product. We have $\text{par-deg}(E_*)\,=\, \text{par-deg}(\det E_*)$. So we have
$\text{par-deg}(E_*)\,=\, \text{par-deg}(\phi_*\det E_*)$. On the other hand,
$\text{par-deg}(\phi_*\det E_*)\,=\, \text{par-deg}(\det \phi_* E_*)$.
\end{proof}

\begin{proposition}\label{propn1}\mbox{}
\begin{enumerate}
\item For any compact Riemann surface $M$ and any nonconstant holomorphic map
$\varphi\, :\, Z\, \longrightarrow\, M$,
$$
(\varphi\circ\phi)_*E_*\,=\, \varphi_*(\phi_* E_*)_*\, .
$$

\item If $\phi$ is a (ramified) Galois morphism, then the parabolic vector bundle
$\phi^*(\phi_*E_*)_*$ is isomorphic to the direct sum
$$
\bigoplus_{\gamma\in {\rm Gal}(\phi)} \gamma^*E_*
$$
of parabolic bundles.
\end{enumerate}
\end{proposition}

\begin{proof}
Both the statements are straightforward to prove. It suffices to prove the second statement only
for parabolic line bundles.
\end{proof}

\begin{proposition}\label{prop4}
A parabolic vector bundle $E_*$ on $X$ is semistable if and only if the parabolic vector
bundle $\phi_*E_*$ on $Z$ is semistable.
\end{proposition}

\begin{proof}
Take a compact connected Riemann surface $M$ and a nonconstant holomorphic map
$\rho\, :\, M\, \longrightarrow\, X$ such that the composition $\phi\circ\rho$
is (ramified) Galois.

First assume that $E_*$ is semistable. So from Lemma \ref{lemn2}(2) we know that
$\rho^*E_*$ is semistable. From Proposition \ref{propn1}(2) we have
$$
(\phi\circ\rho)^*(\phi\circ\rho)_*\rho^*E_*\,=\, \bigoplus_{\gamma\in {\rm Gal}(\phi\circ\rho)}
\rho^*E_*\, .
$$ 
Therefore, the semistability of $\rho^*E_*$ implies that $(\phi\circ\rho)^*(\phi\circ\rho)_*\rho^*E_*$
is semistable.

Note that $(\phi\circ\rho)^*\phi_* E_*$ is
a parabolic subbundle of $(\phi\circ\rho)^*(\phi\circ\rho)_*\rho^*E_*$. Also, they
have the same parabolic slope; this follows from Lemma \ref{lem-n}
and Lemma \ref{lemn2}(1). Since $(\phi\circ\rho)^*(\phi\circ\rho)_*\rho^*E_*$
is semistable, this implies that $(\phi\circ\rho)^*\phi_* E_*$ is semistable. Now from
Lemma \ref{lemn2}(2) it follows that $\phi_* E_*$ is semistable.

To prove the converse, assume that the parabolic bundle $\phi_* E_*$ is semistable.
If $E_*$ is not semistable, let $F_*\, \subset\, E_*$ be the maximal semistable subbundle
of $E_*$ (the first nonzero term of the Harder-Narasimhan filtration of $E_*$). Since
$\mu(F_*)\, >\, \mu(E_*)$, it follows using Lemma \ref{lem-n}
that $\mu(\phi_* F_*)\, >\, \mu(\phi_* E_*)$, and hence
$\phi_* E_*$ is not semistable. In view of this contradiction we conclude that $E_*$ is semistable.
\end{proof}

\section{Pullback of parabolic Higgs bundles and parabolic connections}
\label{section:pullback}

\subsection{Pullback of parabolic Higgs bundles}

As in \eqref{e5}, let $Y$ be a compact connected Riemann surface and
$$
f\, :\, Y\, \longrightarrow\, X
$$
a nonconstant holomorphic map. As in \eqref{e7}, set $B\,=\, f^{-1}(D)_{\rm red}$.
Take a parabolic vector bundle $E_*\,:=\, \left(E,\, \{E^j_i\}, \,\{\alpha^j_i\}\right)$
on $(X,\, D)$. Choose a Higgs field
$$
\theta\, \in\, H^0(X,\, \text{End}_P(E_*)\otimes K_X(D))
$$
on it. So $(f^*\theta)\big\vert_{Y\setminus B}$ is a Higgs field on the holomorphic vector bundle 
$(f^*E)\big\vert_{Y\setminus B}$.

\begin{proposition}\label{prop1}
The Higgs field $(f^*\theta)\big\vert_{Y\setminus B}$ on $(f^*E)\big\vert_{Y\setminus B}$ extends
to a Higgs field on the parabolic vector bundle $f^*E_*$ constructed in Section \ref{se3}.
\end{proposition}

\begin{proof}
First consider the special case where $\text{rank}(E)\,=\, 1$. In this case
$\text{End}_P(E_*)\, =\, {\mathcal O}_X$, and hence
$$
\theta\, \in\, H^0(X,\, K_X(D))\, .
$$
So
$$
f^*\theta\, \in\, H^0(Y,\, K_Y\otimes {\mathcal O}_Y(B))\, ,
$$
because $f^*K_X(D) \,=\, K_Y(B)$.
Therefore, $f^*\theta$ is a Higgs field on the parabolic line bundle $f^*E_*$.

Next consider the case where $E_*$ is a direct sum of two parabolic line bundles.
In other words,
$$
E_*\,\,=\,\, L_*\oplus M_*\, ,
$$
where $L_*$ and $M_*$ are parabolic line bundles with parabolic structure over $D$. The holomorphic
line bundle underlying $L_*$ (respectively, $M_*$) will be denoted by $L$ (respectively, $M$).
Take a parabolic point $x\, \in\, D$. Let $\alpha$ (respectively, $\beta$) be the parabolic
weight of $L_*$ (respectively, $M_*$) at $x$. Take a point $y\, \in\, f^{-1}(x)$, and let $m\, \geq\, 1$
be the multiplicity of $f$ at $y$.

If $\beta\,=\, \alpha$, then it is easy to see that $f^*\theta$ extends, across the point $y$, to a Higgs field on
the parabolic bundle $f^*E_*$ around $y$.

Assume that
\begin{equation}\label{e18}
\beta\, >\, \alpha.
\end{equation}
First consider the case where $\lfloor{m\alpha}\rfloor\,=\,
\lfloor{m\beta}\rfloor$. Then from \eqref{e18} we have
\begin{equation}\label{e17}
m\beta - \lfloor{m\beta}\rfloor \,\, >\,\, m\alpha - \lfloor{m\alpha}\rfloor\, .
\end{equation}
Take a Higgs field
$$
\theta\, \in\, H^0(X,\, \text{End}_P(E_*)\otimes K_X(D))
$$
on $E_*\,=\, L_*\oplus M_*$. Consider the homomorphism
$$
\theta(x)\, :\, 
L_x\oplus M_x \, \longrightarrow\, (L_x\oplus M_x)\otimes K_X(D)_x \,=\,
L_x\oplus M_x\, ;
$$
recall from \eqref{pa} that $K_X(D)_x\,=\, {\mathbb C}$. From \eqref{e18} it follows that the composition
of homomorphisms
\begin{equation}\label{e19}
M_x \, \hookrightarrow\, L_x\oplus M_x \, \xrightarrow{\,\, \theta(x)\,\,\,} L_x\oplus M_x\,
\longrightarrow\, L_x
\end{equation}
is the zero homomorphism (see \eqref{e3}).

Let $\widetilde{E}$ denote the holomorphic vector bundle underlying the parabolic vector bundle
$f^*E_*$. From the construction of $f^*E_*$ we have
$$
\widetilde{E}_y\,=\, (L_x\otimes {\mathcal O}_Y(\lfloor{m\alpha}\rfloor y)_y)
\oplus (M_x\otimes {\mathcal O}_Y (\lfloor{m\beta}\rfloor y)_y)\,=\, (L_x\oplus M_x)\otimes 
{\mathcal O}_Y(\lfloor{m\alpha}\rfloor y)_y
$$
(recall that $\lfloor{m\alpha}\rfloor\,=\, \lfloor{m\beta}\rfloor$). Since the composition of homomorphisms
in \eqref{e19} vanishes, it now follows that $f^*\theta$ extends, across the point $y$, to a Higgs field on
the parabolic bundle $f^*E_*$ around $y$.

Next consider the case where $$\lfloor{m\alpha}\rfloor\,<\, \lfloor{m\beta}\rfloor\, .$$
Now we have
$$
\widetilde{E}_y\,=\, (L_x\otimes {\mathcal O}_Y(\lfloor{m\alpha}\rfloor y)_y)
\oplus (M_x\otimes {\mathcal O}_Y (\lfloor{m\beta}\rfloor y)_y)
$$
$$
=\, (L_x\oplus (M_x\otimes
{\mathcal O}_Y ((\lfloor{m\beta}\rfloor-\lfloor{m\alpha}\rfloor)y)_y))\otimes 
{\mathcal O}_Y(\lfloor{m\alpha}\rfloor y)_y\, .
$$

If $m\beta - \lfloor{m\beta}\rfloor \, >\, m\alpha - \lfloor{m\alpha}\rfloor$, then the argument
for the previous case works. On the other hand, if
$$m\beta - \lfloor{m\beta}\rfloor \,\, \leq\,\, m\alpha - \lfloor{m\alpha}\rfloor\, ,$$
then note that the natural homomorphism
$$
{\mathcal O}_Y\,\, \hookrightarrow\,\, {\mathcal O}_Y ((\lfloor{m\beta}\rfloor-\lfloor{m\alpha}\rfloor)y)
$$
has the property that the homomorphism between the fibers over $y$
$$
{\mathcal C}\,=\, ({\mathcal O}_Y)_y\, \longrightarrow\,
({\mathcal O}_Y ((\lfloor{m\beta}\rfloor-\lfloor{m\alpha}\rfloor)y)_y
$$
vanishes. From this it follows that $f^*\theta$ extends, across the point $y$, to a Higgs field on
the parabolic bundle $f^*E_*$ around $y$.

Therefore, for any Higgs field $\theta\, \in\, H^0(X,\, \text{End}_P(E_*)
\otimes K_X(D))$ on $E_*$, the pullback $f^*\theta$ produces Higgs field on $f^*E_*$.

For the case of higher rank parabolic bundles, the above argument generalizes in a straightforward way.
\end{proof}

Proposition \ref{prop1} has the following straightforward consequence:

\begin{corollary}\label{cor1}
If $(E_*,\, \theta)$ is a strongly parabolic Higgs bundle, then $(f^*E_*,\, f^*\theta)$ is also
a strongly parabolic Higgs bundle.
\end{corollary}

\begin{remark}\label{lemma:unique}
The extension described in Proposition \ref{prop1} is evidently the unique extension of $f^*\theta$ to a parabolic Higgs bundle
on $f^*E_*$ over $(Y,\,B)$.
\end{remark}

\begin{lemma}\label{lem2}
A parabolic Higgs bundle $(E_*,\, \theta)$ on $X$ is semistable if and only if
the parabolic Higgs bundle $(f^*E_*,\, f^*\theta)$ is semistable.
\end{lemma}

\begin{proof}
The proof of Lemma \ref{lemn2}(2) goes through once the obvious modifications to it are made.
\end{proof}

Take any parabolic Higgs bundle $(E_*,\, \theta)$ of rank $r$ on $X$. The Hitchin map $\mathcal H$ sends it
to $\bigoplus_{i=1}^r \text{trace}(\theta^i)\, \in\, H^0(X,\, K_X(D)^{\otimes i})$ (see \cite{Hi}). We
have the natural pullback map
$$
f^*_i\,:\, H^0(X,\, K_X(D)^{\otimes i})\, \longrightarrow\, H^0(Y,\, K_Y(B)^{\otimes i})\, .
$$

The following lemma is straightforward.

\begin{lemma}\label{lemh}
For any parabolic Higgs bundle $(E_*,\, \theta)$ of rank $r$ on $X$,
$$
\left(\bigoplus_{i=1}^r f^*_i\right)\circ {\mathcal H}(E_*,\, \theta)\,=\, {\mathcal H}(f^*E_*,\, f^*\theta)\, ,
$$
where $\mathcal H$ denotes the Hitchin map.
\end{lemma}

\subsection{Pullback of parabolic connections}\label{section:pullbackConnection}

Take a holomorphic vector bundle $V$ on $X$, and let
$$
{\nabla}\,\,:\,\, V\, \longrightarrow\, V\otimes K_X(D)
$$
be a logarithmic connection on $V$ singular over $D$. Then $f^*\nabla$ is a logarithmic connection $f^* V$
singular over $B\,=\, f^{-1}(D)_{\rm red}$, where $f$ is a holomorphic map as in \eqref{e5}. For $x_i\, \in\, D$, if
$${\rm Res}({\nabla},\, x_i)\, \in\, \text{End}(V_{x_i})$$ is the residue of $\nabla$ at $x_i$, then
the residue of $f^*\nabla$ at any $y_{i,j}\,\in\, f^{-1}(x_i)$ is
$$
m_{i,j}\cdot {\rm Res}({\nabla},\, x_i)\, \in\, \text{End}((f^*V)_{y_{i,j}})\,=\,
\text{End}(V_{x_i})\, ,
$$
where $m_{i,j}$, as before, denotes the multiplicity of $f$ at the point $y_{i,j}$. 

We will now describe another property of the residue that will be useful here.

Take a connected Riemann surface $M$ together with a point $x\, \in\, M$. Let $V$ be a holomorphic vector
bundle on $M$ and 
$$
{\nabla}\,\,:\,\, V\, \longrightarrow\, V\otimes K_M(x)\,=\, V\otimes K_M\otimes {\mathcal O}_M(x)
$$
a logarithmic connection on $V$ singular over $x$. Take a subspace
$$
H\,\, \subset\,\, V_x
$$
such that ${\rm Res}({\nabla},\, x)(H)\, \subset\, H$. Let $W$ be the holomorphic vector bundle on $M$
defined by the following short exact sequence of coherent analytic sheaves on $M$:
$$
0\, \longrightarrow\, W \, \longrightarrow\, V \, \longrightarrow\, V_x/H \, \longrightarrow\, 0\, .
$$
Let $\delta\, :\, W_x \, \longrightarrow\, V_x$ be the corresponding homomorphism of fibers over $x$. The
composition of homomorphisms
$$
W\, \longrightarrow\, W_x \, \stackrel{\delta}{\longrightarrow}\, V_x
$$
will be denoted by $\widetilde\delta$. We also have the short exact sequence of coherent analytic sheaves on $M$
$$
0\, \longrightarrow\, V\otimes {\mathcal O}_M(-x) \, \longrightarrow\, W \,
\stackrel{\widetilde\delta}{\longrightarrow}\, \widetilde{\delta}(W) \, \longrightarrow\, 0\, .
$$
Let $\eta\, :\, (V\otimes {\mathcal O}_M(-x))_x \, \longrightarrow\, W_x$
be the corresponding homomorphism of fibers over $x$. So we have an exact sequence
\begin{equation}\label{ee}
0\, \longrightarrow\, H\otimes ({\mathcal O}_M(-x))_x \, \longrightarrow\,
(V\otimes {\mathcal O}_M(-x))_x \, \stackrel{\eta}{\longrightarrow}\, W_x
\, \stackrel{\delta}{\longrightarrow}\, V_x \, \longrightarrow\, V_x/H \, \longrightarrow\, 0\, .
\end{equation}

The logarithmic connection ${\nabla}$ on $V$ produces a logarithmic connection on $V\otimes {\mathcal O}_M(-x)$;
this logarithmic connection on $V\otimes {\mathcal O}_M(-x)$ will be denoted by ${\nabla}'$. We have
$$
{\rm Res}({\nabla}',\, x)\,=\, {\rm Res}({\nabla},\, x) +\text{Id}_{V_x}\, ;
$$
note that $\text{End}((V\otimes {\mathcal O}_M(-x))_x)\,=\, \text{End}(V_x)$. So the subspace
$$
H\otimes ({\mathcal O}_M(-x))_x \, \hookrightarrow\,
(V\otimes {\mathcal O}_M(-x))_x
$$
in \eqref{ee} is preserved by ${\rm Res}({\nabla}',\, x)$; recall that ${\rm Res}({\nabla},\, x)(H)\, \subset\, H$.
The logarithmic connection ${\nabla}$ on
$V$ induces a logarithmic connection on $W$; this induced logarithmic connection on $W$ will be denoted
by ${\nabla}^W$. The residue ${\rm Res}({\nabla}^W,\, x)$ preserves the image of $\eta$ in \eqref{ee}, and
the restriction of ${\rm Res}({\nabla}^W,\, x)$ to the image of $\eta$ coincides with the endomorphism
induced by ${\rm Res}({\nabla}',\, x)$ (as noted above, $H\otimes ({\mathcal O}_M(-x))_x$ is preserved by
${\rm Res}({\nabla}',\, x)$), and hence ${\rm Res}({\nabla}',\, x)$ produces an endomorphism of the image of $\eta$.
Also, the action of ${\rm Res}({\nabla}^W,\, x)$ on the
image of $\delta$ in \eqref{ee} coincides with the action of ${\rm Res}({\nabla},\, x)$ on $S$.

Let $\nabla$ be a parabolic connection on a parabolic vector bundle $E_*$ on $X$. Consider the holomorphic
connection on $(f^*E)\big\vert_{Y\setminus B}$ obtained by pulling back the connection $\nabla$ using the
map $f$ in \eqref{e5}, where $B$ is defined in \eqref{e7}. This connection actually extends to a connection
on the parabolic vector bundle $f^*E_*$. Indeed, this is straightforward to check using the construction
of the parabolic vector bundle $f^*E_*$ and the above mentioned properties of residue of a logarithmic connection. Moreover,
this extension is unique (as in Remark \ref{lemma:unique}).

We can now prove the following:

\begin{theorem}\label{thm1}
Take any $f:Y\longrightarrow X$ as in \eqref{e5}. A parabolic vector bundle $E_*$ on $X$ is polystable if and only
if the parabolic vector bundle $f^*E_*$ on $Y$ is polystable.
\end{theorem}

\begin{proof}
The parabolic endomorphism bundle $\text{End}_P(E_*)$ in \eqref{e3} has a parabolic structure
induced by the parabolic structure of $E_*$. In fact, this parabolic bundle $\text{End}_P(E_*)_*$ is
the parabolic tensor product $E_*\otimes E^*_*$, where $E^*_*$ is the parabolic dual of $E_*$.

It can be shown that the parabolic vector bundle $E_*$ is polystable if and only if $\text{End}_P(E_*)$
is polystable. To prove this first assume that $E_*$ is polystable. Then $E_*$ has a unique unitary
projectively flat complex connection \cite{Biq} (see also \cite{MS}). The connection on $E_*\otimes E^*_*$
induced by this connection on $E_*$ is unitary flat. Hence $E_*\otimes E^*_*\,=\, \text{End}_P(E_*)_*$ is
polystable.

Next assume that $\text{End}_P(E_*)_*$ is polystable. We will first show that $E_*$ is semistable. If $E_*$
is not semistable, let $F_*\, \subset\, E_*$ be the first nonzero term of the Harder--Narasimhan filtration of
$E_*$ (so $F_*$ is the maximal semistable subsheaf of $E_*$). Now the parabolic subbundle
$$
\text{Hom}(E_*,\, F_*)_*\, =\, F_*\otimes E^*_*\, \subset\, E_*\otimes E^*_*\,=\, \text{End}(E_*)_*
$$
contradicts the semistability condition for $\text{End}_P(E_*)_*$. Hence $E_*$ is semistable. Let
$V_*\, \subset\, E_*$ be the unique maximal parabolic polystable subbundle, of same parabolic
slope, of the parabolic semistable bundle
$E_*$ \cite[p.~24, Lemma 1.5.5]{HL}, so $V_*$ is the socle of $E_*$. Assume that $V*\, \not=\, E_*$.
Since $\text{End}_P(E_*)_*$ is polystable, its subbundle
$$
\text{Hom}(E_*,\, V_*)_*\, =\, V_*\otimes E^*_*\, \subset\, E_*\otimes E^*_*\,=\, \text{End}(E_*)_*
$$
is also polystable. Fix a subbundle $W_*\, \subset\, \text{End}(E_*)_*$ such that
$\text{End}(E_*)_*\,=\, \text{Hom}(E_*,\, V_*)_*\oplus W_*$. Now $W_*(V_*)$ is a polystable subbundle of
$E_*$ and it contradicts the maximality of the socle $V_*$. Therefore, we conclude that $E_*$ is polystable.

To prove the theorem first assume that $E_*$ is polystable. As shown above, this implies that
$\text{End}_P(E_*)_*$ has a unitary flat holomorphic connection. On the other hand, we have
$$
f^*\text{End}_P(E_*)_*\,=\, \text{End}_P(f^*E_*)\, ,
$$
because the pullback operation is compatible with tensor product, direct sum and dualization operations
on parabolic bundles. So the unitary flat holomorphic connection on $\text{End}_P(E_*)_*$ pulls back to
a unitary flat holomorphic connection on $\text{End}_P(f^*E_*)_*$. This implies that $\text{End}_P(f^*E_*)_*$
is polystable. Hence $f^*E_*$ is polystable.

To prove the converse, assume that $f^*E_*$ is polystable. If $f$ is not (ramified) Galois, choose another
nonconstant holomorphic surjective map $g\, :\, M\, \longrightarrow\, Y$ such that $f\circ g$ is
(ramified) Galois. Since $f^*E_*$ is polystable, we know that $g^*f^*E_*\,=\, (f\circ g)^*E_*$ is polystable.
Let $\nabla$ denote the unique unitary projectively flat complex connection on $(f\circ g)^*E_*$. We note that
$\nabla$ is preserved by the natural action of the Galois group $\text{Gal}(f\circ g)$ on $(f\circ g)^*E_*$.

Since $(f\circ g)^*E_*$ is semistable, from Lemma \ref{lemn2}(2) we know that $E_*$ is semistable. Let
$$
F_*\, \subset\, E_*
$$
be the socle, meaning the unique maximal polystable subbundle of same parabolic
slope; see \cite[p.~23, Lemma 1.5.5]{HL}. Consider the parabolic subbundle
$$
(f\circ g)^*F_*\, \subset\, (f\circ g)^*E_*\, ;
$$
it is polystable because $F_*$ is so. Let $G_*\, \subset\, (f\circ g)^*E_*$ be the orthogonal complement of
$(f\circ g)^*F_*$ for the Hermitian structure on $(f\circ g)^*E_*$ (to which the unique unitary projectively
flat complex connection is associated).

Since both $(f\circ g)^*F_*$ and $\nabla$ are preserved by the natural
action of the Galois group $\text{Gal}(f\circ g)$ on $(f\circ g)^*E_*$, we conclude that
$G_*\, \subset\, (f\circ g)^*E_*$ is also preserved by the natural
action of the Galois group $\text{Gal}(f\circ g)$ on $(f\circ g)^*E_*$. Therefore, there is a parabolic
subbundle $G'_*\, \subset\, E_*$ such that the two parabolic subbundles $G_*$ and
$(f\circ g)^*G'_*$ of $(f\circ g)^*E_*$ coincide. We now have
\begin{equation}\label{f1}
E_*\,=\, F_*\oplus G'_*
\end{equation}
because $(f\circ g)^*E_*\,=\, (f\circ g)^*F_*\oplus G_*$. But \eqref{f1} contradicts the fact that
$F_*$ is the unique maximal polystable subbundle of $E_*$ of same parabolic slope. This is because the
direct sum of $F_*$ with the socle of $G'_*$ (if $G'_*$ is nonzero) is also polystable. Therefore,
we conclude that $E_*$ is polystable.
\end{proof}

The following proposition is proved using Theorem \ref{thm1}.

\begin{proposition}\label{prop3}
Take $\phi$ as in \eqref{ep}. For any polystable parabolic bundle $E_*$ on
$X$ the parabolic bundle $\phi_*E_*$ on $Z$ is polystable.
\end{proposition}

\begin{proof}
The proof is identical to the proof in
Proposition \ref{prop4} with semistability replaced by polystability.
\end{proof}

\section{Direct image of parabolic Higgs bundles and connections}
\label{section:directimage}

\subsection{Direct image of parabolic Higgs bundles}
\label{section:directHiggs}

Take a nonconstant holomorphic map $\phi\, :\, X\, \longrightarrow\, Z$ between compact connected
Riemann surfaces (as in \eqref{ep}). As in \eqref{e8}, $R\, \subset\, X$ is the ramification locus of $\phi$. Let
$$
\widehat{R}\,=\, \phi(R)
$$
be its image in $Z$. Let $\phi^{-1}(\widehat{R})_{\rm red}$ be the reduced inverse image of $R$, so
we have $R\, \subset\, \phi^{-1}(\widehat{R})_{\rm red}$.

It can be shown that there is a natural injective homomorphism of coherent analytic sheaves
\begin{equation}\label{e16}
\Phi\,\,:\,\, \phi_*K_X \, \longrightarrow\, K_Z\otimes {\mathcal O}_Z(\widehat{R})\otimes \phi_*{\mathcal O}_X
\end{equation}
which is an isomorphism over the complement $Z\setminus\widehat{R}$.
To see this first note that
$$\phi^* (K_Z\otimes {\mathcal O}_Z(\widehat{R}))\,=\,
K_X\otimes{\mathcal O}_X(\phi^{-1}(\widehat{R})_{\rm red})\, .$$
Therefore, the projection formula (see \cite[p.~124, Ex.~5.1(d)]{Ha}) gives that 
\begin{equation}\label{e17a}
\phi_* (K_X\otimes{\mathcal O}_X(\phi^{-1}(\widehat{R})_{\rm red}))\,=\,
K_Z\otimes {\mathcal O}_Z(\widehat{R})\otimes \phi_*{\mathcal O}_X\, .
\end{equation}
But $\phi_*K_X\, \subset\, \phi_* (K_X\otimes{\mathcal O}_X(\phi^{-1}(\widehat{R})_{\rm red}))$ as
$K_X\, \subset\, K_X\otimes{\mathcal O}_X(\phi^{-1}(\widehat{R})_{\rm red})$, and hence from \eqref{e17a}
we get a homomorphism $\Phi$ as in \eqref{e16}. Since
$$
(\phi_*K_X)\big\vert_{Z\setminus \widehat{R}}\,=\,
(\phi_*(K_X\otimes{\mathcal O}_X(\phi^{-1}(\widehat{R})_{\rm red})))\big\vert_{Z\setminus\widehat{R}},
$$
the homomorphism $\Phi$ is an isomorphism over $Z\setminus\widehat{R}$.

Let $E_*$ be a parabolic vector bundle on $X$ with parabolic structure over $D$. Take a parabolic Higgs field
$$\theta\, :\, E\, \longrightarrow\, E\otimes K_X(D)$$ on $E_*$. The direct image of it is a holomorphic
homomorphism
\begin{equation}\label{e20a}
\phi_*\theta\, :\, \phi_*E\, \longrightarrow\, \phi_*(E\otimes K_X(D)).
\end{equation}
Evidently, we have $D\, \subset\, \phi^{-1}(\Delta)_{\rm red}$, where $\Delta\, \subset\, Z$ is the divisor in \eqref{e9}.
So $K_X(D)\, \subset\, K_X(\phi^{-1}(\Delta)_{\rm red})$, and hence it follows that
\begin{equation}\label{e20d}
E\otimes K_X(D)\, \subset\, E\otimes K_X(\phi^{-1}(\Delta)_{\rm red})\,=\, E\otimes \phi^* K_Z(\Delta)\, .
\end{equation}
Consequently, from the projection formula we have
\begin{equation}\label{e20c}
\phi_*(E\otimes K_X(D))\, \subset\, (\phi_* E)\otimes K_Z(\Delta)\, .
\end{equation}
Therefore, $\phi_*\theta$ in \eqref{e20a} gives a holomorphic homomorphism
\begin{equation}\label{e20b}
\phi_*\theta\, :\, \phi_*E\, \longrightarrow\, (\phi_* E)\otimes K_Z(\Delta)\, .
\end{equation}

It is straightforward to check that $\phi_*\theta$ in \eqref{e20b} is a
parabolic Higgs field on the parabolic vector bundle $\phi_*E_*$.

\begin{lemma}\label{ln1}
If $\phi$ is a (ramified) Galois morphism, then for any parabolic Higgs bundle $(E_*,\, \theta)$
on $X$, the pulled back strongly parabolic Higgs bundle 
$(\phi^*(\phi_*E_*)_*,\, \phi^*\phi_*\theta)$ is isomorphic to the direct sum
$$
\bigoplus_{\gamma\in {\rm Gal}(\phi)} (\gamma^*E_*,\, \gamma^*\theta)
$$
of parabolic Higgs bundles.
\end{lemma}

\begin{proof}
{}From Proposition \ref{propn1}(2) we know that
$$
\phi^*(\phi_*E_*)_*\,=\,\bigoplus_{\gamma\in {\rm Gal}(\phi)} \gamma^*E_*.
$$
On $X\setminus \phi^{-1}(\Delta)_{\rm red}$ (see \eqref{e9}) we have
$$
(\phi^*\phi_* \theta)\big\vert_{X\setminus \phi^{-1}(\Delta)_{\rm red}}\,=\,
\bigoplus_{\gamma\in {\rm Gal}(\phi)} (\gamma^*\theta)\big\vert_{X\setminus \phi^{-1}(\Delta)_{\rm red}}.
$$
Therefore, $(\phi^*(\phi_*E_*)_*,\, \phi^*\phi_*\theta)$ is
isomorphism to
$\bigoplus_{\gamma\in {\rm Gal}(\phi)} (\gamma^*E_*,\, \gamma^*\theta)$.
\end{proof}

\begin{proposition}\label{pr1}
A parabolic Higgs bundle $(E_*,\, \theta)$ on $X$ is semistable if and only if
$(\phi_*E_*,\, \phi_*\theta)$ is semistable.
\end{proposition}

\begin{proof}
In view of Lemma \ref{lem2} and Lemma \ref{ln1}, the proof of Proposition \ref{prop4} gives a proof
after the obvious modifications are made.
\end{proof}

There is a natural homomorphism
$$P\, :\, \phi_*{\mathcal O}_X\, \longrightarrow\, {\mathcal O}_Z$$ that simply sends any holomorphic
function $\beta$ on $\phi^{-1}(U)$, where $U\, \subset\, Z$ is an open subset, to the function on $U$ whose value at
any $u\, \in\, U$ is $\sum_{x \in \phi^{-1}(u)} \beta(x)\, \in\, \mathbb C$; here $\phi^{-1}(u)$ denotes the inverse image
with multiplicities. Therefore, for any holomorphic vector bundle $V$ on $Z$, we have a natural homomorphism
\begin{equation}\label{eg}
\phi_*\phi^* V\,=\, V\otimes \phi_*{\mathcal O}_X\, \stackrel{P}{\longrightarrow}\, V\otimes {\mathcal O}_Z \,=\, V\, .
\end{equation}

Consider $\Delta$ in \eqref{e9}. Recall that
$$
K_X(D)\, \subset\, \phi^* K_Z(\Delta)
$$
(see \eqref{e20d}). Hence $K_X(D)^{\otimes i}\, \subset\, \phi^* K_Z(\Delta)^{\otimes i}$ for all $i\, \geq\, 1$.
Therefore, from \eqref{eg} we have a homomorphism
$$
\phi_*(K_X(D)^{\otimes i}) \, \subset\, \phi_*(\phi^* K_Z(\Delta)^{\otimes i})\,\longrightarrow\,
K_Z(\Delta)^{\otimes i}\, .
$$
Let
$$
\Psi_i\,\,:\,\,\phi_*(K_X(D)^{\otimes i})\, \longrightarrow\,K_Z(\Delta)^{\otimes i}
$$
be this homomorphism.

\begin{lemma}\label{lemh2}
For any parabolic Higgs bundle $(E_*,\, \theta)$ of rank $r$ on $X$,
$$
\left(\bigoplus_{i=1}^r \Psi_i\right)\circ {\mathcal H}(E_*,\, \theta)\,=\, {\mathcal H}(f_*E_*,\, f_*\theta)\, ,
$$
where $\mathcal H$ is the Hitchin map as in Lemma \ref{lemh}.
\end{lemma}

\begin{proof}
This follows immediately from the above construction of $\Psi_i$ and the definition of the Hitchin map.
\end{proof}

\subsection{Direct image of parabolic connections}
\label{section:directConn}

Let
$$
{\nabla}\,\,:\,\, E\, \longrightarrow\, E\otimes K_X(D)
$$
be a connection on $E_*$. The direct image of ${\nabla}$ for the map $\phi$ produces a holomorphic
differential operator
\begin{equation}\label{e20}
\phi_*{\nabla}\,\, :\,\,\phi_*E \,\longrightarrow\, \phi_*(E\otimes K_X(D)) .
\end{equation}
Now using \eqref{e20c}, the differential operator in \eqref{e20} gives a holomorphic differential operator
\begin{equation}\label{e21}
\phi_*{\nabla}\,\, :\,\,\phi_*E \,\longrightarrow\, (\phi_* E)\otimes K_Z(\Delta)\, .
\end{equation}

It is straightforward to check that the differential operator $\phi_*{\nabla}$ in \eqref{e21}
defines a connection on the parabolic vector bundle $\phi_*E$ constructed in Section \ref{se4}.

\begin{proposition}\label{prop:semistableConn}
A parabolic connection $(E_*,\nabla)$ on $X$ is semistable if and only if $(\phi_*E_*,\phi_*\nabla)$ is semistable.
\end{proposition}

\begin{proof}
The proof is analogous to the proof of Proposition \ref{prop4} and Proposition \ref{pr1}.
\end{proof}

\section{Compatibility with nonabelian Hodge theory}\label{section:naht}

Let us start by recalling the left-continuous filtration formalism for parabolic bundles used by Simpson for the noncompact nonabelian Hodge 
correspondence \cite{Si}. Given a parabolic bundle $(E,\,\{E_i^j\},\,\{\alpha_i^j\})$ we can define a collection of vector bundles $E_i^\alpha$ for 
each $i$ and $\alpha\in \mathbb{R}$ as follows:
\begin{itemize}
\item If $\alpha\,=\,\alpha_i^j$, then $E_i^{\alpha_i^j}$ is the subsheaf of $E$ that fits in the short exact sequence
$$0\,\longrightarrow\, E_i^{\alpha_i^j} \,\longrightarrow\, E \,\longrightarrow\, E_x/E_i^j \,\longrightarrow\, 0.$$

\item If $\alpha_i^j \,\le\, \alpha \,< \,\alpha_i^{j+1}$ or $\alpha_i^n \,\le\, \alpha\,<\,\alpha_i^1+1$. then $E_i^\alpha\,=\,E_i^{\alpha_i^j}$.

\item $E_i^{\alpha+1}\,=\,E_i^\alpha (-x_i)$.
\end{itemize}
The last equation can be also used to extend the definition of $E_i^\alpha$ for $\alpha\,<\,0$. If $$j_i\,:\,U_i
\,=\,X\backslash\{x_i\} \,\hookrightarrow\, X$$ is the inclusion of complement of the parabolic point in $X$ into $X$, then clearly the
vector bundles $E_i^\alpha$ provide a left continuous decreasing filtration of the quasi-projective sheaf
$$\bigcup_{\alpha\in \mathbb{R}} E_i^\alpha = (j_i)_* E|_{U_i} =: E(\infty \cdot x_i).$$

Under this formalism, a parabolic Higgs bundle can be described as an ${\mathcal O}_X$-linear map
$$\theta: E\longrightarrow E\otimes K_X(D)$$
such that
$$\theta(E_i^\alpha)\,\subseteq\, E_i^\alpha\otimes K_X(D) \quad\, \forall\ \ \alpha \,\in\, \mathbb{R}$$
and a parabolic connection can be described as a map
$$\nabla: E\longrightarrow E\otimes K_X(D)$$
satisfying the Leibniz rule \eqref{e4} together with the condition 
$$\nabla(E_i^\alpha)\,\subseteq\, E_i^\alpha\otimes K_X(D) \quad \,\forall\ \ \alpha \,\in\, \mathbb{R}.$$

On the other hand, given a holomorphic vector bundle $E$ on $X$ and a parabolic point $x_i\in D$, suppose that $E|_{X\backslash D}$
is given an acceptable metric $K$ \cite{Si}, in the sense that the curvature $R_K$ of the metric connection of $K$ satisfies the
following bound around $x_i$:
$$|R_K| \,\le \,f + \frac{C}{r^2(\log r)^2}\, ,$$
where $r$ is the radial distance function from $x_i$ and $f$ is some $L^p$ function. Then, by \cite[Proposition 3.1]{Si}, the metric $K$
induces a filtration of subsheaves of $E(\infty\cdot x_i)$ as follows: $E_i^\alpha$ is the subsheaf which coincides with $E$ over $U_i
\,=\,X\backslash\{x_i\}$, but whose stalk at $x_i$ is formed by sections $e$ of $E$ over a punctured disk around $x_i$ which satisfy the growth condition
$$|e|_K\,\le\, Cr^{\alpha-\varepsilon} \ \ \quad \forall\ \ \varepsilon\,>\,0.$$

When the metric $K$ is the metric of a tame harmonic bundle, these sheaves $E^\alpha_i$ coincide with the parabolic filtrations of the
corresponding parabolic Higgs bundle or parabolic connection. Following \cite{Si}, recall that a harmonic bundle is quadruple
$(\mathcal{E},\,\mathcal{D}_\theta,\,\mathcal{D}_\nabla,\,K)$ consisting of
\begin{itemize}
\item a $\mathcal{C}^\infty$-vector bundle $\mathcal E$ on $U\,=\,X\backslash D$,

\item differential operators $\mathcal{D}_\theta,\mathcal{D}_\nabla\,:\,\mathcal{E}\,\longrightarrow\, \mathcal{E}\otimes \Omega_U^{1,1}$, and

\item a hermitian metric $K$ on $\mathcal{E}$
\end{itemize}
such that the following conditions are satisfied:
\begin{enumerate}
\item $\mathcal{D}_\theta$ splits into $(0,\,1)$ and $(1,\,0)$ parts as $\mathcal{D}_\theta\,=\,\overline{\partial}_\theta+\theta$,
where $\overline{\partial}_\theta$ is a holomorphic structure on $\mathcal{E}$ and $\theta$ is a holomorphic Higgs field on $(\mathcal{E},
\,\overline{\partial}_\theta)$.

\item $\mathcal{D}_\nabla$ splits into $(0,\,1)$ and $(1,\,0)$ parts as $\mathcal{D}_\nabla\,=\,\overline{\partial}_\nabla+\nabla$, where
$\overline{\partial}_\nabla$ is a holomorphic structure on $\mathcal{E}$ and $\nabla$ is a holomorphic connection on
$(\mathcal{E},\,\overline{\partial}_\nabla)$.

\item Moreover, $\overline{\partial}_\nabla \,=\, \overline{\partial}_\theta+\overline{\theta}$ and $\nabla\,=\,\partial_\theta+
\overline{\theta}$, where
$$(\partial_\theta u,\, v)_K+(u,\, \overline{\partial}_\theta v)_K\,=\,\partial(u,\,v)_K$$
$$(u,\,\theta v)\,=\,(\overline{\theta}u,\,v)$$
for each pair of local sections $u$ and $v$ of $\mathcal{E}$.

\item The curvature and pseudo-curvature of the metric are zero: $\mathcal{D}_\theta^2\,=\,\mathcal{D}_\nabla^2\,=\,0$.
\end{enumerate}

Equivalently, \cite[\S~1]{Si}, the data of a harmonic bundle is equivalent to providing a representation of the fundamental group 
$\pi_1(U)\longrightarrow \operatorname{GL}(r,\mathbb{C})$ together with an equivariant harmonic map $\widetilde{U}\,\longrightarrow\,
\operatorname{GL}(r,\mathbb{C})/{\rm U}(r)$, where $\widetilde{U}$ is the universal cover of $U\,=\,X\backslash D$ and $r$ is the rank of the bundle.

A harmonic bundle is tame if the singularities of the Higgs field (or equivalently, of the connection) are at most logarithmic. In that case, the metric $K$ will be acceptable \cite[Theorem 4]{Si}.

Given a tame harmonic bundle $(\mathcal{E},\,\mathcal{D}_\theta,\,\mathcal{D}_\nabla,\,K)$, let $E_{\theta,i}^\alpha$ for $\alpha\,\in\,
\mathbb{R}$ 
be the set of sheaves indexed by $\mathbb{R}$ obtained by applying the previous construction to the holomorphic bundle 
$(\mathcal{E},\,\overline{\partial}_\theta)$ and the metric $K$ around the point $x_i$. Analogously, let $E_{\nabla,i}^\alpha$ be the sheaves 
obtained applying the previous construction to the holomorphic bundle $(\mathcal{E},\,\overline{\partial}_\nabla)$ and the metric $K$ around $x_i$. 
Then define $E_\theta$ and $E_\nabla$ as the holomorphic vector bundles obtained gluing together the bundles $E_{\theta,i}^0$ and 
$E_{\nabla,i}^0$ for all parabolic points respectively.

By Simpson's correspondence \cite[p.~755, Main Theorem]{Si}, there is an equivalence between the categories of tame harmonic bundles, direct 
sums of stable parabolic Higgs bundles and direct sums of stable parabolic connections. In particular, for each algebraic stable parabolic 
Higgs bundle $(E,\,E_*,\,\theta)$ with $E\,=\,(\mathcal{E},\,\overline{\partial}_\theta)$ there exists a harmonic metric $K$ and a stable parabolic 
connection $(E',\,E'_*,\,\nabla)$ on $E'\,=\,(\mathcal{E},\,\overline{\partial}_\nabla)$ such that $(\mathcal{E}|_U,\, 
\overline{\partial}_\theta+\theta,\,\overline{\partial}_\nabla+\nabla,\,K)$ is a tame harmonic bundle. Analogously, any parabolic connection 
extends to a compatible tame harmonic bundle.

Let us prove that these constructions are compatible with the pullbacks and direct images described earlier.

\begin{theorem}
\label{thm:pullbackNAHT}
Let $(X,D)$ be a marked curve. Let $f\,:\,Y\,\longrightarrow\, X$ be a nonconstant holomorphic map of Riemann surfaces, and
let $B\,=\,f^{-1}(D)_{\operatorname{red}}\,\subset \,Y$. Let $(E,\,E_*,\,\theta)$ be a parabolic Higgs bundle, and let $(E',\,E'_*,\,\nabla)$
be a parabolic connection on $(X,\,D)$ induced by the same tame harmonic bundle $(\mathcal{E},\,\mathcal{D}_\theta,\,\mathcal{D}_\nabla,
\,K)$. Then the pullback $(f^*\mathcal{E},\,f^*\mathcal{D}_\theta,\,f^*\mathcal{D}_\nabla,\,f^*K)$ is a tame harmonic bundle on $(Y,\,B)$
giving a correspondence between the pullbacks $f^*(E,\,E_*,\,\theta)$ and $f^*(E',\,E'_*,\,\nabla)$ defined in Section \ref{section:pullback}.
\end{theorem}

\begin{proof}
Denote $U\,=\,X\backslash D$ and $V\,=\,Y\backslash B$. Let $$\rho\,:\,\pi_1(U)\,\longrightarrow\, \operatorname{GL}(r,\mathbb{C})\ \,
\text{ and }\,\ 
\widetilde{K}\,:\,\widetilde{U}\,\longrightarrow\, \operatorname{GL}(r,\mathbb{C})/{\rm U}(r)$$ respectively be the representation of the fundamental group
and the harmonic map to the symmetric space associated to $(\mathcal{E},\,\mathcal{D}_\theta,\,\mathcal{D}_\nabla,\,K)$. The map $f$
induces a homomorphism $\pi_1(f)\,:\,\pi_1(V)\,\longrightarrow\, \pi_1(U)$ and a holomorphic map
$\widetilde{f}\,:\widetilde{V}\,\longrightarrow\, \widetilde{U}$. Then the composition
$\rho\circ \pi_1(f)\,:\,\pi_1(V)\,\longrightarrow \,\operatorname{GL}(r,\mathbb{R})$ is a representation of $\pi_1(V)$ and, as the
composition of a harmonic map with a holomorphic map is harmonic, the map $\widetilde{K}\circ \widetilde{f}\,:\,
\widetilde{V}\,\longrightarrow\, \operatorname{GL}(r,\mathbb{C})/{\rm U}(r)$ is harmonic. It is then clear by construction that the pair $(\rho\circ \pi_1(f),\, \widetilde{K}\circ
\widetilde{f})$ correspond to the pullback $(f^*\mathcal{E},\,f^*\mathcal{D}_\theta,\,f^*\mathcal{D}_\nabla,\,f^*K)$, so it is a
harmonic bundle.

If $\mathcal{D}_\theta
\,=\,\overline{\partial}_\theta+\theta$ and $\mathcal{D}_\nabla\,=\,\overline{\partial}_\nabla + \nabla$ are the splittings in $(0,1)$ and $(1,0)$ parts of the operators of the original harmonic bundle, it is clear that
$$f^*\mathcal{D}_\theta \,=\, f^*\overline{\partial}_\theta + f^*\theta \, , \quad \quad f^*\mathcal{D}_\nabla
\,=\,f^*\overline{\partial}_\nabla+ f^*\nabla\, ,$$
so
$$(f^*\mathcal{E},\,f^*\overline{\partial}_\theta)\,=\,f^*(\mathcal{E},\,\overline{\partial}_\theta)\, ,
\quad \quad (f^*\mathcal{E},\,f^*\overline{\partial}_\nabla)\,=\,f^*(\mathcal{E},\,\overline{\partial}_\nabla)\, .$$

Moreover, by \cite[Theorem 2]{Si}, the metric $f^*K$ is acceptable for $\overline{\partial}_\theta$, so we can analyze the
filtration induced by the metric $K$ on the vector bundle $f^*(\mathcal{E},\,\overline{\partial}_\theta)$ (it is also acceptable for
$\overline{\partial}_\nabla$ and the computation for the connection will be analogous) at a point $y_{i,j}\,
\in\, f^{-1}(x_i)\,\subset\, B$ of multiplicity $m_{i,j}$. Let $x$ be a local coordinate around $x_i$, and $y$ be a local coordinate
around $y_{i,j}$, so that the map $f$ is locally described as $x\,=\,y^{m_{i,j}}$ around $y_{i,j}$. Any local section $e'$ of
$f^*(\mathcal{E},\,\overline{\partial}_\theta)$ in a punctured disc around $y_{i,j}$ can then be described as
$$e'\,=\,\sum_{k=-\infty}^\infty y^k\cdot f^*e_k\, ,$$
where $e_k$ is a local section of $(\mathcal{E},\,\overline{\partial}_\theta)$ defined on
a punctured neighborhood of $x_i$. Let $r_Y$ denote the radius around $y_{i,j}$ and $r_X$ denote the radius around $x_i$. Then
$$|e'|_{f^*K} \,\le\, \sum_{k=-\infty}^\infty r_Y^k |e_k|_K\, .$$
In particular, the sections $e'$ with
$$|e'|_{f^*K}\,\le\, Cr_Y^{\alpha-\varepsilon}\,=\,Cr_X^{\frac{\alpha-\varepsilon}{m_{i,j}}}\,=\,Cr_X^{\frac{\alpha}{m_{i,j}}-\varepsilon'}$$
correspond to the sections of the form $e'\,=\,\sum_k y^k\cdot f^*e_k$ such that $r_Y^k|e_k|_K \,\le\,
Cr_X^{\frac{\alpha}{m}-\varepsilon'}$ for each $k$. We have
$$\left\{r_Y^k|e_k|_K \,\le \,Cr_X^{\frac{\alpha}{m_{i,j}}-\varepsilon'}\ \, \forall\ \varepsilon'\,>\,0\right\}\ \quad \Leftrightarrow\ \quad
\left\{|e_k|_K \,\le\, Cr_X^{\frac{\alpha-k}{m_{i,j}}-\varepsilon'}\ \,\forall \varepsilon'\,>\,0\right\}
$$
$$
\Leftrightarrow\ \quad
\left\{e_k \,\in\, (\mathcal{E},\,\overline{\partial}_\theta)_i^{\frac{\alpha-k}{m_{i,j}}}\right\}.$$
Thus, we obtain:
$$\left(f^*(\mathcal{E},\,\overline{\partial}_\theta)\right)_{y_{i,j}}^\alpha \,=\, \sum_{k=-\infty}^\infty \mathcal{O}_Y(-ky_{i,j})
\otimes f^*(\mathcal{E},\,\overline{\partial}_\theta)_i^{\frac{\alpha-k}{m_{i,j}}}\, .$$
In particular, the jumps of the filtration occur precisely at points $\alpha$ where
$$\frac{\alpha-k}{m_{i,j}}-\alpha_i^t \in \mathbb{Z}\, \text{ for some } k\in \mathbb{Z}\, .$$
Simplifying, these are the points $\alpha$ of the form
$$\alpha = m_{i,j}\alpha_i^t + k, \quad k\in \mathbb{Z}\, .$$
In particular, the set of jumps of the filtration between 0 and 1 is
$$\left \{m_{i,j}\alpha_i^t-\lfloor m_{i,j}\alpha_i^t\rfloor\right\}$$
and the extension of the vector bundle $f^*(\mathcal{E},\,\overline{\partial}_\theta)$ to a local neighborhood of $y_{i,j}$ induced by $K$
is then isomorphic to
$$\left(f^*(\mathcal{E},\,\overline{\partial}_\theta)\right)_{y_{i,j}}^0 \,=\, \sum_{k=-\infty}^\infty \mathcal{O}_Y(-ky_{i,j})
\otimes f^*\left ((\mathcal{E},\,\overline{\partial}_\theta)_{x_i}^{-\frac{k}{m_{i,j}}}\right)
$$
$$
=
\,\sum_{t=1}^{n_i} \mathcal{O}_Y(\lfloor m_{i,j} \alpha_i^t \rfloor y_{i,j})\otimes 
f^*\left ((\mathcal{E},\,\overline{\partial}_\theta)_{x_i}^{\alpha_i^t}\right)\, ,$$
due to the following identities which are consequences of the 1-periodicity and the left continuous nature of the filtration:
\begin{multline*}
\mathcal{O}_Y(-(k+m_{i,j})y_{i,j})\otimes f^*\left ((\mathcal{E},\,\overline{\partial}_\theta)_{x_i}^{-\frac{k+m_{i,j}}{m_{i,j}}}\right)=\mathcal{O}_Y(-(k+m_{i,j})y_{i,j})\otimes f^*\left ((\mathcal{E},\,\overline{\partial}_\theta)_{x_i}^{-\frac{k}{m_{i,j}}-1}\right)\\
=\mathcal{O}_Y(-(k+m_{i,j})y_{i,j})\otimes f^*\left ((\mathcal{E},\,\overline{\partial}_\theta)_{x_i}^{-\frac{k}{m_{i,j}}}(x_i)\right)\\
=\mathcal{O}_Y(-(k+m_{i,j})y_{i,j})\otimes f^*\left ((\mathcal{E},\,\overline{\partial}_\theta)_{x_i}^{-\frac{k}{m_{i,j}}}\right)\otimes \mathcal{O}_Y(m_{i,j}y_{i,j})\\
=\mathcal{O}_Y(-ky_{i,j})\otimes f^*\left ((\mathcal{E},\,\overline{\partial}_\theta)_{x_i}^{-\frac{k}{m_{i,j}}}\right)
\end{multline*}
and, if $\frac{k}{m_{i,j}}\le \alpha_i^a<\ldots <\alpha_i^b <\frac{k+1}{m_{i,j}}$, then
\begin{multline*}
\mathcal{O}_Y(-ky_{i,j}) \otimes (\mathcal{E},\,\overline{\partial}_\theta)_{x_i}^{\frac{k}{m_{i,j}}} = \mathcal{O}_Y(-ky_{i,j}) \otimes (\mathcal{E},\,\overline{\partial}_\theta)_{x_i}^{\alpha_i^a}\\
= \mathcal{O}_Y(-ky_{i,j}) \otimes \sum_{t=a}^b (\mathcal{E},\,\overline{\partial}_\theta)_{x_i}^{\alpha_i^t}= \sum_{t=a}^b \mathcal{O}_Y(-ky_{i,j}) \otimes (\mathcal{E},\,\overline{\partial}_\theta)_{x_i}^{\alpha_i^t}\, .
\end{multline*}
Splitting the filtered bundle as a sum of parabolic line bundles it is then straightforward to check that this bundle coincides with the 
pulled back parabolic bundle $f^*(E,\, E_*)$ (defined in \eqref{eq:parpullback} and \eqref{eq:parpullback2}) on a 
neighborhood of $y_{i,j}$. Combining all these we obtain the following: The parabolic bundle on $(Y,\, B)$ induced by $K$, as an extension of 
$f^*(\mathcal{E}, \,\overline{\partial}_\theta)$, is exactly the pulled back parabolic bundle $f^*E_*$ described in Section \ref{se3}. The
result for the holomorphic bundle associated to the connection is analogous: We obtain that the extension of 
$f^*(\mathcal{E},\,\overline{\partial}_\nabla)$ to a parabolic bundle on $(Y,\,K)$ induced by $K$ is exactly $f^*(E',\,E'_*)$.

By Proposition \ref{prop1}, the the holomorphic Higgs bundle
$(f^*\mathcal{E},\,f^*\overline{\partial}_\theta,
\,f^*\theta)$ extends to a Higgs field on the parabolic vector bundle $f^*E_*$ on $Y$. Since such an extension is unique by
Remark \ref{lemma:unique}, the filtered Higgs bundle induced by $f^*(\mathcal{E},\,\mathcal{D}_\theta,\,\mathcal{D}_\nabla,\,K)$
must be the pullback $f^*(E,\,E_*,\,\theta)$ described in Section \ref{section:pullback}. Moreover, as $f^*\theta$ has regular
singularities, the harmonic bundle $(f^*\mathcal{E},\,f^*\overline{\partial}_\theta,f^*\theta)$ is tame.

Analogously, the parabolic connection $f^*\nabla$ described in Section \ref{section:pullbackConnection} is the unique extension to $f^*E_*$ 
of the holomorphic connection $(f^*\mathcal{E},\,f^*\overline{\partial}_\nabla,\,f^*\nabla)$.
\end{proof}

We can verify that the residue diagram prescribed by the construction in Section \ref{section:pullbackConnection} corresponds through 
\cite[Theorem 7]{Si} to the one which is associated to $f^*(E,\,E_*,\,\theta)$.

Let $y_{i,j}\,\in\, f^{-1}(x_i)$ be a parabolic point in $Y$ of multiplicity $m_{i,j}$. Fix a local
holomorphic coordinate $y$ of $Y$ around $y_{i,j}$ and also
a local holomorphic coordinate function $x$ of $X$ around $x_i$, so that $x\,=\,y^{m_{i,j}}$. Suppose that, locally around $x$,
the Higgs field $\theta$ is written as $\theta\,=\,\frac{A}{x}dx$, for some matrix $A$, and $\nabla$ is written as
$\nabla\,=\,\frac{A'}{x}dx+d$ for some matrix $A'$. Computing the residue map of $\theta$ or $\nabla$ is then the same as computing
the action of $x\frac{\partial}{\partial x}$, which clearly yields $\text{Res}(\theta,\,x_i)\,=\,A$ and $\text{Res}(\nabla,\,x_i)\,=\,A'$.

Observe that for each local section of the form $y^k(f^*e)(y)\,=\,y^ke(y^{m_{i,j}})$, we have
\begin{equation}
\label{eq:partial}
y\frac{\partial}{\partial y} f^*e \,=\, y\frac{\partial}{\partial y} e(y^{m_{i,j}}) \,=\,
m_{i,j}y^{m_{i,j}}\frac{\partial}{\partial x}e(x) \,=\, m_{i,j} x \frac{\partial}{\partial x} e(x).
\end{equation}
Computing the residues of $f^*\theta$ and $f^*\nabla$ then reduces to computing the action of 
$y\frac{\partial}{\partial y}$ on local sections of the form $y^{-k}(f^*e_k)
(y)\,=\,y^{-k}e_k(y^{m_{i,j}})$. For the $\mathcal{O}_Y$-linear operator $\theta$, this is then equivalent to computing the action of $mx\frac{\partial}{\partial x}$ on the corresponding local sections $e_k$ on $X$. Thus,
$$\text{Res}(f^*\theta,\,y_{i,j})\,=\,m_{i,j}\text{Res}(\theta,\,x_i)\,=\,m_{i,j}A.$$
In the case of the parabolic connection $f^*\nabla$ we have
$$\nabla_{y\frac{\partial}{\partial y}} (y^{-k}f^*e_k) \,= \, y^{-k} \nabla_{y\frac{\partial}{\partial y}}(f^*e_k) -k y^k f^*e_k .$$
Therefore, for each $e\,\in \,E_i^{\alpha_i^l}$ we have
$$\text{Res}(f^*\nabla,\,y_{i,j})(y^{-\lfloor m_{i,j}\alpha_i^l\rfloor}f^*e|_{y_{i,j}})
\,=\,\left (m_{i,j}A'-\lfloor m_{i,j} \alpha_i^l \rfloor I\right)(y^{-\lfloor m_{i,j}\alpha_i^l\rfloor}f^*e|_{y_{i,j}}) .$$

As we know the way the eigenvalues of $A$ and $A'$ are related through nonabelian Hodge correspondence \cite{Si}, we can then compute Table \ref{table:pullback} summarizing how the jumps and eigenvalues change for the parabolic Higgs bundle and the parabolic connection when the pullback is taken.

\begin{table}[h]
\begin{center}
\begin{tabular}{|c|c|c|c|c|}
\hline
&\textbf{$(E,E_*,\theta)$}&\textbf{$(E',E'_*,\nabla)$}&\textbf{$f^*(E,E_*,\theta)$}&\textbf{$f^*(E',E'_*,\nabla)$}\\
\hline
\textbf{Jumps} & $\alpha$ & $\alpha-2b$ & $m_{i,j}\alpha$ & $m_{i,j}(\alpha-2b)$ \\
\hline
\textbf{Eigenvalues} & $b+ci$ & $\alpha +2ci$ & $m_{i,j}b+m_{i,j}ci$ & $m_{i,j}\alpha+ 2m_{i,j}ci $\\
\hline
\end{tabular}
\end{center}
\caption{Relations between the jumps and eigenvalues of corresponding parabolic Higgs bundles and connections through nonabelian Hodge theory and pullbacks.}
\label{table:pullback}
\end{table}

Coherently with the 1-periodicity conditions on parabolic Higgs bundles and connections, the table is given considering the weights $\alpha \pmod{\mathbb{Z}}$. If we replace $\alpha$ by $\alpha+k$ at the jump of a connection, then we replace the corresponding eigenvalue from
$b'+c'i$ by $b'+k+c'i$.

\begin{remark}
\label{rmk:pullback}
Observe that, from this table, it is clear that the pullback sends strongly parabolic Higgs bundles ($b\,=\,c\,=\,0$) to strongly parabolic
Higgs bundles. On the connection side, this corresponds to the pullback taking connections whose residue has the same eigenvalues
as the weights to connections with the same property. As this correspondence also preserves the nilpotent part of the action of the residue
on the graded vector space with respect to the filtration, this means that the pullback preserves the ``residual'' condition described
following Definition \ref{dlc}.
\end{remark}

\begin{theorem}
\label{thm:directNAHT}
Let $(X,\,D)$ be a marked curve. Let $\phi\,:\,X\,\longrightarrow\, Z$ be a nonconstant holomorphic map of compact Riemann surfaces, and
$\Delta\,\subset\, Z$ denotes the image of the union of the ramification points and the parabolic points (see \eqref{e9}).
Take a parabolic Higgs bundle $(E,\,E_*,\,\theta)$ and a parabolic connection $(E',E'_*,\nabla)$ on $(X,\,D)$
induced by a single tame harmonic bundle $(\mathcal{E},\,\mathcal{D}_\theta,\,\mathcal{D}_\nabla,\, K)$. Then the restriction of
the direct image $(\phi_*\mathcal{E}|_{Z\backslash \Delta},\,\phi_*\mathcal{D}_\theta|_{Z\backslash \Delta},\,
\phi_*\mathcal{D}_\nabla|_{Z\backslash \Delta},\,\phi_*K|_{Z\backslash \Delta})$ is a tame harmonic bundle on $(Z,\,\Delta)$ giving a
correspondence between the direct images $\phi_*(E,\,E_*,\,\theta)$ and $\phi_*(E',\,E'_*,\,\nabla)$
defined in Section \ref{section:directimage}.
\end{theorem}

\begin{proof}
Fix a decomposition $\phi^{-1}(Z\backslash \Delta)\,=\,U_1\cup\ldots\cup U_d$ into a disjoint union of open subsets of $X$. In order to
simplify notation, denote $\mathcal{F}\,:=\,\mathcal{E}|_{\phi^{-1}(Z\backslash \Delta)}$. Locally over $Z\backslash \Delta$, the direct image 
$(\phi_*\mathcal{F},\,\phi_*\mathcal{D}_\theta,\,\phi_*\mathcal{D}_\nabla,\,\phi_*K)$ is just a direct sum of the restrictions of the harmonic 
bundle $(\mathcal{F},\,\mathcal{D}_\theta,\,\mathcal{D}_\nabla,\,K)$ to each $U_i$, $i\,=\,1,\,\ldots,\,d$. As the conditions (1)--(4) in the 
definition of a harmonic bundle are local, the restriction of the direct image to $Z\backslash \Delta$ must be a harmonic bundle and 
the metric $\phi_*K|_{Z\backslash \Delta}$ is acceptable \cite[Theorem 2]{Si}.

We decompose
$$\phi_*\mathcal{D}_\theta \,=\, \phi_*\overline{\partial}_\theta + \phi_*\theta \quad \text{and} \quad \phi_*\mathcal{D}_\nabla \,=\,
\phi_*\overline{\partial}_\nabla + \phi_* \nabla\, .$$
We will analyze the filtration of $(\phi_*\mathcal{F},\,
\phi_*\overline{\partial}_\theta)$ induced by $\phi_*K|_{Z\backslash \Delta}$ around a point $z\,\in\, \Delta$. 
Let $V$ be a small disc around $x$ such that $\phi^{-1}(V)\,=\,V_1'\cup \ldots\cup V_{b_z}'$ is a disjoint union of simply connected
open subsets with $\phi^{-1}(x)\cap V_i \,=\, \{y_i\}$ for some point $y_i\,\in\, X$ and $\# \phi^{-1}(x')\bigcap V_i \, = \, m_i$
for any other $x'\,\ne \,x$. Each local holomorphic section of the direct image on a punctured disc around $x$ is of the form
$(e_1,\,\ldots,\,e_{b_z})$, where $e_i$ is a local section in a punctured neighborhood of $y_i$ in $V_i$. If $r_Z$ is the radial
distance function from $x$ and $r_{V_i}$ is the radial distance function from each $y_i$, then
$$
\left\{|(e_1,\,\ldots,\, e_{b_z})|_{\phi_*K} \,\le \, C r_Z^{\alpha-\varepsilon}\ \, \forall \varepsilon\,>\,0\right\}\, \quad \Leftrightarrow
$$
$$
\left\{|e_i|_K \,\le\, C r_Z^{\alpha-\varepsilon}\,=\,Cr_{V_i}^{m_i(\alpha-\varepsilon)}\,=\,
Cr_{V_i}^{m_i\alpha-\varepsilon'} \,\, \forall \varepsilon'\,>\,0\ \forall i\right\}\,\quad
\Leftrightarrow \quad \left\{e_i \,\in\, (\mathcal{F},\,\overline{\partial}_\theta)_{y_i}^{m_i\alpha}\,\, \forall i\right\}.
$$
Thus,
$$(\phi_*\mathcal{F},\,\phi_*\overline{\partial}_\theta)_x^\alpha \,=\, \bigoplus_{i=1}^{b_z} (\mathcal{F},\,
\overline{\partial}_\theta)_{y_i}^{m_i\alpha}\, .$$
The filtration of each component $(\mathcal{F},\,\overline{\partial}_\theta)_{y_i}^{m_i\alpha}$ jumps precisely when
$$m_i\alpha - \alpha_i^j \,=\,k\,\in\, \mathbb{Z}\, ,$$
so it jumps precisely at the points of the form
$$\alpha\,=\,\frac{\alpha_i^j+k}{m_i}$$
for some $i\,=\,1,\,\ldots,\,b_z$ and $k\,\in\, \mathbb{Z}$. As $0\,\le\, \alpha_i^j \,<\,1$, the set of these jumps between $0$ and
$1$ corresponds to taking $k\,=\,0,\,\ldots,\,m-1$ for each $i$. Moreover, observe that
\begin{multline*}
(\mathcal{F},\,\overline{\partial}_\theta)_{y_i}^{m_i\left(\alpha+\frac{1}{m_i}\right)}\,=\,
(\mathcal{F},\,\overline{\partial}_\theta)_{y_i}^{m_i\alpha}(-x)\,\supset\, \ldots \,\supset\,
(\mathcal{F},\,\overline{\partial}_\theta)_{y_i}^{m_i\left(\alpha+\frac{k}{m_i}\right)}
\,=\,(\mathcal{F},\,\overline{\partial}_\theta)_{y_i}^{m_i\alpha}(-kx)\\
\supset\, \ldots \,\supset\, (\mathcal{F},\,\overline{\partial}_\theta)_{y_i}^{m_i\left(\alpha+1\right)}
\,=\,(\mathcal{F},\,\overline{\partial}_\theta)_{y_i}^{m_i\alpha}(-m_ix)
\end{multline*}
for each $\alpha$ and $y_i$. This is precisely the structure of the direct image $\phi_*(E,\,E_*)$ constructed in \eqref{e13}, so we
conclude that the parabolic bundle on $(Z,\,\Delta)$ induced by $(\phi_*\mathcal{F},\,\phi_*\overline{\partial}_\theta,\,\phi_*K)$ is
precisely $\phi_*(E,\,E_*)$. Analogously, the parabolic bundle induced by $(\phi_*\mathcal{F},\,\phi_*\overline{\partial}_\nabla,\,
\phi_*K)$ is $\phi_*(E',\,E'_*)$.

We can now proceed analogously to Theorem \ref{thm:pullbackNAHT}. There is at most one possible extension of $\phi_*\theta$ from 
$E|_{Z\backslash \Delta}$ to a parabolic Higgs bundle on $\phi_*(E,\,E_*)$ on $Z$. As the direct image construction $\phi_*(E,E_*,\theta)$ 
described in Section \ref{section:directHiggs} provides such extension, we conclude that 
$(\phi_*\mathcal{F},\,\phi_*\mathcal{D}_\theta,\,\phi_*\mathcal{D}_\nabla,\,\phi_*K)$ must be tame and the induced filtered Higgs bundle on $Y$ by 
the tame harmonic bundle is $\phi_*(E,\,E_*,\,\theta)$. Analogously, there is at most one parabolic connection extending 
$(\phi_*\mathcal{F},\,\phi_*\overline{\partial}_\nabla,\, \phi_*\nabla)$ to $\phi_*(E',\,E'_*)$. As the direct image 
$\phi_*(E',\,E'_*,\,\nabla)$ constructed in Section \ref{section:directConn} provides such extension, it must be the one induced by the 
harmonic bundle.
\end{proof}

Finally, we can compute the residues to find an analogous table to Table \ref{table:pullback}. Take $y \,\in\, \Delta$. Decompose
\begin{equation}
(\phi_* E)_y\,=\, \bigoplus_{x_i\in \phi^{-1}(y)} V_{x_i}
\end{equation}
as in \eqref{e10}. The residue of both $\theta$ and $\nabla$ will clearly be diagonal with respect to this decomposition, so it
suffices to compute the residues of $\theta$ and $\nabla$ on $V_{x_i}$ for each $x_i\,\in\,
\phi^{-1}(y)$. Take a local holomorphic coordinate function $x$ around $x_i$, and let $m\,=\,
m_{x_i}$ be the multiplicity of $\phi$ at $x_i$. Write locally $\theta$ and $\nabla$ up to elements of order $x^m$:
$$\theta\,=\,\left(\frac{A_{-1}}{x}+A_0+A_1x+\ldots+A_{m-1}x^{m-1}\right) \otimes dx\, $$
$$\nabla\,=\,\left(\frac{A_{-1}'}{x} +A_0'+A_1'x+\ldots+A_{m-1}'x^{m-1}\right) \otimes dx +d\, .$$
Let $v\,=\,v_0+xv_1+\ldots+x^{m-1}v_{m-1}$ be a local section of $E$ up to order $m$. Then, 
$$\theta_{x\frac{\partial}{\partial x}}(v)=\sum_{k=0}^{m-1} \sum_{j=0}^{k} A_{j-1} v_{k-j} x^k+O(x^m)$$
$$\nabla_{x\frac{\partial}{\partial x}}(v)=\sum_{k=0}^{m-1} \sum_{j=0}^{k} A'_{j-1} v_{k-j} x^k+O(x^m)\, .$$
Separating locally the bases in the $1,x,\ldots,x_{m-1}$ blocks and taking into account \eqref{eq:partial} yield
$$\text{Res}(\phi_*\theta,y)|_{V_{x_i}} = \frac{1}{m} \left(\begin{array}{c|c|c|c}
A_{-1} & 0 & \cdots & 0\\
\hline
A_0 & A_{-1} & \ddots & 0\\
\hline
\vdots & \ddots& \ddots & \vdots\\
\hline
A_{m-2} & A_{m-2} & \cdots & A_{-1} 
\end{array}\right)\, ,$$
$$\text{Res}(\phi_*\theta,y)|_{V_{x_i}} = \frac{1}{m}\left(\begin{array}{c|c|c|c}
A_{-1}' & 0 & \cdots & 0\\
\hline
A_0' & A_{-1}'+I & \ddots & 0\\
\hline
\vdots & \ddots& \ddots & \vdots\\
\hline
A_{m-2}' & A_{m-2}' & \cdots & A_{-1}'+(m-1)I 
\end{array}\right)\, .$$
From this computation and the previous discussion we can then derive Table \ref{table:direct} summarizing the relations
between the jumps and eigenvalues of the previous objects and their direct images.

\begin{table}[h]
\begin{center}
\begin{tabular}{|c|c|c|c|c|}
\hline
&\textbf{$(E,E_*,\theta)$}&\textbf{$(E',E'_*,\nabla)$}&\textbf{$\phi_*(E,E_*,\theta)$}&\textbf{$\phi_*(E',E'_*,\nabla)$}\\
\hline
\textbf{Jumps} & $\alpha$ & $\alpha-2b$ & $\left \{\frac{\alpha_j}{m_j}+\frac{k}{m_j} \right\}$ & $\left \{\frac{\alpha_j-2b_j}{m_j}+\frac{k}{m_j} \right\}$ \\
\hline
\textbf{Eigenvalues} & $b+ci$ & $\alpha +2ci$ & $\left \{ \frac{b_j}{m_j}+\frac{c_j}{m_j} i\right\}$ & $\left \{ \frac{\alpha_j}{m_j}+\frac{k}{m_j}+\frac{c_j}{m_j}i\right\}$\\
\hline
\end{tabular}
\end{center}
\caption{Relations between the jumps and eigenvalues of corresponding parabolic Higgs bundles and connections through nonabelian Hodge theory and direct images.}
\label{table:direct}
\end{table}

As in Table \ref{table:pullback}, the weights are considered $\pmod{\mathbb{Z}}$ and replacing $\alpha$ by $\alpha+k$ at the jump of a 
connection results in replacing by $b'+k+c'i$ all the corresponding eigenvalues $b'+c'i$.

\begin{remark}
\label{rmk:direct}
Analogously to Remark \ref{rmk:pullback}, from the table, we conclude that the direct image sends strongly parabolic Higgs bundles to strongly parabolic Higgs bundles and ``residual'' connections to ``residual'' connections, as described by the comment after Definition \ref{dlc}.
\end{remark}

\section*{Acknowledgements}

David Alfaya would like to thank Carlos Simpson for helpful discussions about the 
nonabelian Hodge correspondence. He is supported by MICINN grant PID2019-108936GB-C21 and
the second-named author is partially supported by a J. C. Bose Fellowship.


\end{document}